\documentclass[a4paper]{amsart}
\usepackage[plain]{fullpage}
\usepackage{amsmath}
\usepackage{amssymb}
\usepackage{xcolor}
\usepackage{bookmark}
\usepackage[cmtip,all]{xy}
\usepackage{enumitem}

\makeatletter
\newcommand{\raisemath}[1]{\mathpalette{\raisem@th{#1}}}
\newcommand{\raisem@th}[3]{\raisebox{#1}{$#2#3$}}
\makeatother

\hypersetup{ linktoc=page }

\def\resp{{\sfcode`\.1000 resp.}}
\def\ie{{\sfcode`\.1000 i.e.}}
\def\eg{{\sfcode`\.1000 e.g.}}
\def\cf{{\sfcode`\.1000 cf.}}
\def\st{{\sfcode`\.1000 s.t.}}
\def\loccit{{\sfcode`\.1000 loc.\hspace{1pt}cit.}}
\def\opcit{{\sfcode`\.1000 op.\hspace{1pt}cit.}}

\setitemize{leftmargin=0.75cm}

\renewcommand{\AA}{\mathbb{A}}
\newcommand{\ZZ}{\mathbb{Z}}
\newcommand{\PP}{\mathbb{P}}
\newcommand{\GG}{\mathbb{G}}

\newcommand{\D}{\mathcal{D}}
\renewcommand{\S}{\mathcal{S}}
\newcommand{\I}{\mathcal{I}}

\newcommand{\Sm}{\mathrm{S}\mathrm{m}}
\newcommand{\Sch}{\mathrm{S}\mathrm{ch}}

\newcommand{\sPre}{\mathrm{s}\mathrm{P}\mathrm{re}}

\newcommand{\ob}{{ob}}
\newcommand{\op}{{op\hspace{-2pt}}}
\newcommand{\simpl}{{s\hspace{-2pt}}}
\newcommand{\Aeins}{{\raisemath{-1pt}{\AA^{\hspace{-1pt}1}\hspace{-1pt}}}}
\newcommand{\adjoint}{\rightleftarrows}

\newcommand{\pt}{{\mbox{\larger[-10]$+$}\hspace{-1pt}}}
\newcommand{\bigslant}[2]{{\raisebox{0em}{$#1$}/\raisebox{-.5em}{$#2$}}}

\DeclareMathOperator{\colim}{colim}

\DeclareMathOperator{\Hom}{Hom}
\DeclareMathOperator{\Spt}{Spt}
\DeclareMathOperator{\SH}{{\mathcal{S}\mathcal{H}}}

\DeclareMathOperator{\hocolim}{hocolim}
\DeclareMathOperator{\holim}{holim}
\DeclareMathOperator{\hhom}{\underline{\hom}}
\DeclareMathOperator{\Pic}{Pic}
\DeclareMathOperator{\Sing}{{\mathit{Sing}}}
\DeclareMathOperator{\Spec}{{Spec}}

\newcounter{zaehler}
\theoremstyle{plain}

\newtheorem{introtheorem}[zaehler]{Theorem}

\newtheorem{theorem}{Theorem}[section]
\newtheorem{proposition}[theorem]{Proposition}
\newtheorem{corollary}[theorem]{Corollary}
\newtheorem{question}[theorem]{Question}
\newtheorem{lemma}[theorem]{Lemma}
\theoremstyle{definition}
\newtheorem{reduction}[theorem]{Reduction}
\newtheorem{defi}[theorem]{Definition}
\newtheorem{example}[theorem]{Example}
\newtheorem{remark}[theorem]{Remark}

\numberwithin{equation}{section}
\title{Stable $\AA^1$-connectivity over Dedekind schemes}
\author{Johannes Schmidt}
\address[Johannes Schmidt]{\newline Mathematisches Institut, Universit\"at Heidelberg, 69120 Heidelberg, Germany}
\email{jschmidt@mathi.uni-heidelberg.de}
\author{Florian Strunk}
\address[Florian Strunk]{\newline Fakult\"at f\"ur Mathematik, Universit\"at Regensburg, 93040 Regensburg, Germany}
\email{florian.strunk@ur.de}

\thanks{The authors are supported by the SFB/CRC 1085 \emph{Higher Invariants} (Regensburg) funded by the DFG and the DFG-Forschergruppe 1920 
\emph{Symmetrie, Geometrie und Arithmetik} (Heidelberg--Darmstadt)}

\begin{document}

\begin{abstract}
We show that $\AA^1$-localization decreases the stable connectivity by at most one over a Dedekind scheme with infinite residue fields. For the proof, we establish a version of Gabber's geometric presentation lemma over a henselian discrete valuation ring with infinite residue field.
\end{abstract}

\maketitle

\section*{Introduction}

\subsection*{Background} 
In \cite{Morel05}, Morel formulated the following property on a scheme $S$, called the \emph{stable $\AA^1$-connectivity property}:
\medskip
\begin{center}
\begin{minipage}{100ex}
\it
The $\AA^1$-lo\-cali\-zation of a connected spectrum on the smooth Nisnevich-site over $S$ is still connected. 
\end{minipage}
\end{center}
\medskip
Here, the notion of connectivity refers to the associated Nisnevich homotopy sheaves or equivalently to the connectivity of the Nisnevich stalks.
Further, in \opcit, he proves this property for $S=\Spec(k)$ where $k$ is a field.
This celebrated result is known as the \emph{stable $\AA^1$-connectivity theorem} \cite[Thm.~6.1.8]{Morel05} and has diverse implications:
Most of the content from Morel's monograph \cite{Morel12} is based on this result as for example the unstable $\AA^1$-connectivity theorem and its implication, the Hurewicz Theorem in \mbox{$\AA^1$-homotopy} theory \cite[Thm.~6.37]{Morel12}.
This leads to a computation of the $0$-line of the stable homotopy groups of motivic spheres as the Milnor--Witt $K$-theory $K^{\textrm{MW}}_{*}(S)$ of the base $S$ \cite[Cor.~6.43]{Morel12}.
More immediately, the $\AA^1$-connectivity theorem implies the vanishing of the negative lines which is analogous to the vanishing of the negative stable homotopy groups of the sphere in topology.

In \cite[Conj.~2]{Morel05}, Morel conjectured that the stable $\AA^1$-connectivity property holds over every regular base.
However, in \cite{Ayoub06}, Ayoub constructed a counterexample to this conjecture (see Remark~\ref{counterexample} below).

\subsection*{Aim and results}
In this paper, we want to replace Morel's stable $\AA^1$-connectivity property by the following weaker property on a base scheme $S$ of Krull-dimension $d$ which is consistent with Ayoub's counterexample:
\medskip
\begin{center}
\begin{minipage}{100ex}
\it
The $\AA^1$-lo\-cali\-zation of a $d$-connected spectrum on the smooth Nisnevich-site over $S$ is still connected. 
\end{minipage}
\end{center}
\medskip
We refer to this property as the \emph{shifted stable $\AA^1$-connectivity property}.
In other words, $S$ has this property, if $\AA^1$-localization lowers the connectivity by at most the dimension of $S$.
Question~\ref{thequestion} below asks, if every regular base scheme has this shifted stable \mbox{$\AA^1$-connectivity} property.
Morel's stable connectivity theorem is a positive answer in the case $d=0$. In the main theorem of this paper, we give a positive answer in the one-dimensional case, assuming infinite residue fields (\cf~Theorem~\ref{maintheorem}):

\begin{introtheorem}\label{1maintheorem}
A Dedekind scheme with only infinite residue fields has the shifted stable $\AA^1$-connectivity property:
If $E$ is an $i$-con\-nected spectrum, then its $\AA^1$-localization $L^\Aeins E$ is $(i-1)$-connected.
\end{introtheorem}

Examples for such base schemes are algebraic curves over infinite fields in geometric settings or ${\rm Spec}(\mathbb{Z}_p^{\rm nr})$ for $\mathbb{Z}_p^{\rm nr} / \mathbb{Z}_p$ the maximal unramified extension in more arithmetic settings.

Morel's proof of the $\AA^1$-connectivity theorem needs a strong geometric input refered to as \emph{Gabber's geometric presentation lemma} and written up in \cite[Thm.~3.1.1]{CTHK97}.
In \opcit, the authors show how Gabber's presentation result leads to universal exactness of certain Cousin complexes.
In particular, they derive the Bloch--Ogus theorem and the Gersten conjecture for algebraic $K$-theory for smooth varieties over a field, as first proved by Quillen \cite[Thm.~5.11]{Quillen73}.
In Chapter \ref{chaptergabber}, we prove a version of this presentation result over a henselian discrete valuation ring with infinite residue fields (\cf~Theorem~\ref{thm: gabber distinguished square}):

\begin{introtheorem}\label{thm: gabber distinguished square2}
Let $\mathfrak{o}$ be a henselian discrete valuation ring with infinite residue field and let $\sigma$ denote the closed point of $S=\Spec(\mathfrak{o})$.
Let ${X}$ be a smooth $S$-scheme of finite type and let ${Z}\hookrightarrow {X}$ be a proper closed subscheme.
Let $z$ be a point in ${Z}$.
If $z$ lies in the special fibre $Z_\sigma$, suppose that ${Z}_\sigma \neq {X}_\sigma$.
Then, Nisnevich-locally around $z$, there exists a smooth $\mathfrak{o}$-scheme $V$ of finite type and a cartesian square 
\[
  \xymatrix{
   {X}\setminus{Z} \ar[r] \ar[d] &
   {X} \ar[d]^-p
  \\
   \mathbb{A}_V^1\setminus p({Z}) \ar[r] &
   \mathbb{A}_V^1
  }
\]
such that $p$ is \'{e}tale, the restriction $p|_Z\colon Z\hookrightarrow \mathbb{A}_V^1$ is a closed subscheme and $Z$ is finite over $V$.
In particular, this square is a Nisnevich-distinguished square.
\end{introtheorem}

The proof is based on \cite[Thm.~3.1.1]{CTHK97} combined with a Noether normalization over a Dedekind base (\cf\ \cite[Thm.~4.6]{Kai15}).

Apart from this geometric input to the proof of Theorem~\ref{1maintheorem}, we need a second key ingredient of a more homotopical kind:
In Chapter~\ref{chapterobjectwise}, we examine a vanishing result for the non-sheafified homotopy classes of the $\AA^1$-localization of a connected spectrum. This is a slight generalization of the argument in \cite[Lem.~4.3.1]{Morel05} to arbitrary noetherian base schemes of finite Krull-dimension.
As a byproduct, we obtain that the $S^1$- and the $\PP^1$-homotopy t-structure over any base scheme is left complete, \ie, a presheaf of spectra is recovered as the homotopy limit over its Postnikov truncations (see Corollary \ref{tstructurenondegenerate} and \ref{homotopytstructurenondegenerate}).

\subsection*{Acknowledgement}
We would like to thank Peter Arndt, Joseph Ayoub, Elden Elmanto, Armin Holschbach, Fabien Morel, George Raptis, Markus Spitzweck and Georg Tamme for helpful discussions or comments.
Moreover, we would like to thank the anonymous referee for helpful remarks and comments.

\section{Preliminaries}
\label{sec:preliminaries}

\noindent
In this paper, our base scheme $S$ is always a noetherian scheme of finite Krull-dimension.
Let $\Sm_S$ be the category of smooth schemes of finite type over $S$. The category $\Sm_S$ is essentially small and sometimes we choose a small skeleton implicitly without mentioning.
Let $\sPre_\pt(S)$ be the category of pointed simplicial presheaves on $\Sm_S$.
We will mostly ignore $S$ in the notation.
For an object $U\in \Sm_S$ let $U_\pt$ denote the presheaf $\hom_{\Sm_S}(-,U)$ considered as a discrete simplicial set with an additional disjoint basepoint.
Whenever we speak of a category having all limits and colimits we actually mean that it has all \emph{small} limits and all \emph{small} colimits.
\smallskip

\subsection*{Model structures}
In contrast to the foundational address \cite{MV99} of $\AA^1$-homotopy theory, we use projective analogues of the unstable model structures and obtain the (pointed) \emph{objectwise}, \emph{Nisnevich-local} and \emph{$\AA^1$-Nisnevich-local} model structure (\cf~\cite[Sec.2]{DRO03b}).
Throughout the whole text, let $L^\ob$, $L^\simpl$ and $L^\Aeins$ denote fixed (pointed) objectwise, Nisnevich-local and $\AA^1$-Nisnevich-local fibrant replacement functors, respectively.
For a symbol $\tau\in\{ob, s, \AA^1\}$, a non-negative integer $n$ and $F\in \sPre_\pt$, define the \emph{$n$-th $\tau$-homotopy sheaf} $\pi^\tau_n(F)$ of $F$ as the Nisnevich-sheafification of the \emph{$n$-th $\tau$-homotopy presheaf}
\[
[(-)_\pt\wedge S^n,L^\tau F].
\]
Here, the brackets denote (pointed) objectwise homotopy classes.
Note that $\pi_n^\ob\hspace{1pt}(F)\cong \pi_n^\simpl(F)$.
Whenever the objectwise model structure is considered, we omit the symbol $\phantom{}^\ob~$ from the notation.
\smallskip

Let $\Spt_{S^1}(S)$ be the category of (non-symmetric) $S^1$-spectra on the category $\sPre_\pt(S)$ \cite[Def.~1.1]{Hovey01}.
The functor $(-)_0$ sending an $S^1$-spectrum to its zeroth level and the $S^1$-suspension spectrum functor $\Sigma^\infty_{S^1}$ fit into an adjunction $\Sigma^\infty_{S^1}\colon\sPre_\pt\adjoint \Spt_{S^1}\colon(-)_0$.
For an integer $n\geq 0$, there is also an adjunction $[-n]\colon\Spt_{S^1}\adjoint \Spt_{S^1}\colon[n]$ of shift functors defined for an integer $n$ by $E[n]_m:=E_{n+m}$ whenever $n+m\geq0$ and $E[n]_m:=*$ otherwise.
Following the general procedure of \cite{Hovey01}, we equip the category $\Spt_{S^1}(S)$ with stable model structures (\cf~\cite[Def.~3.3]{Hovey01}) having homotopy categories $\SH_{S^1}^\ob\hspace{2pt}(S)$, $\SH_{S^1}^\simpl\hspace{2pt}(S)$ and $\SH_{S^1}^\Aeins(S)$, respectively.
The two above-mentioned adjunctions turn into Quillen adjunctions.
Each of these stable homotopy categories is a triangulated category with distinguished triangles given by the homotopy cofibre sequences \cite[Prop.~7.1.6]{Hovey99}.
In fact, by choosing symmetric spectra as a more elaborate model, these homotopy categories carry the structure of a closed symmetric monoidal category with a compatible triangulation in the sense of \cite[Def.~4.1]{May01}.
For details of this construction we refer to \cite{Hovey01}, \cite{MSS} and \cite{Ayoub07}.
There are functors
\begin{align*}
-\wedge \Sigma_{S^1}^\infty(-)  \colon & \Spt_{S^1}\times \sPre_\pt\hspace{0.7ex}\to \Spt_{S^1}\\
\hhom(\Sigma_{S^1}^\infty(-),-) \colon &  \sPre_\pt^\op\times \Spt_{S^1}\to \Spt_{S^1}
\end{align*}
defined in the obvious way.
For a cofibrant $F\in\sPre_\pt$, they fit into a Quillen adjunction
\[
-\wedge \Sigma_{S^1}^\infty F \colon \Spt_{S^1}\adjoint\Spt_{S^1} \colon \hhom(\Sigma_{S^1}^\infty F,-)
\]
whose derived adjunction models the monoidal structure from before.

Let $\Omega_{S^1}$ denote the functor $\hhom(\Sigma_{S^1}^\infty S^1,-)$.
We mention that a concrete fibrant replacement functor for the stable $\tau$-model structure on $\Spt_{S^1}$ is given by
\begin{equation}
\label{taufunctor}
\Theta_{S^1}^\tau E=\colim((L^\tau E)\to \Omega_{S^1}(L^\tau E)[1]\to (\Omega_{S^1})^2(L^\tau E)[2]\to\ldots) 
\end{equation}
where $\tau\in\{ob, s, \AA^1\}$ and where the application of $L^\tau$ to a spectrum is levelwise \cite[Thm.~4.12]{Hovey01}.
We write $\Omega^\infty_{S^1}\colon \Spt_{S^1}\to \sPre_\pt$ for the composition of this fibrant replacement functor with $(-)_0$.

As for the unstable structures, we define the \emph{$n$-th stable $\tau$-homotopy sheaf} $\pi^\tau_n(E)$ of a spectrum $E\in \Spt_{S^1}$ as the Nisnevich-sheafification of the \emph{$n$-th stable $\tau$-homotopy presheaf}
\[
[\Sigma_{S^1}^\infty(-_\pt)[n],L^\tau E].
\]
Here, the brackets denote the morphism sets of $\SH^\ob_{S^1}$.
Since it will be evident from the context if the unstable or the stable homotopy sheaf is considered, we do not introduce an extra decoration.

We use the following explicit model for $L^\Aeins$ in the stable context introduced in \cite[Lem~4.2.4]{Morel04}.

\begin{lemma}[Morel]\label{specialreplacement}
Let $S$ be an arbitrary base scheme.
For each integer $k\geq 0$, we set $L^k (E) := \hhom(F^{\wedge k}, L^s (E))$ with $F:=\Sigma^\infty_{S^1}C[-1]$ where $C$ is a cofibrant replacement of the cofibre of the morphism 
\[S^0 \xrightarrow{0,1} \AA^1\]
in $\sPre_\pt$.
Then the functor $L^{\infty}\colon \Spt_{S^1}\to \Spt_{S^1}$ defined by
\[
L^\infty (E) := \underset{k\to \infty}{\hocolim}~ L^k  (E)
\]
is a fibrant replacement functor for the stable $\AA^1$-Nisnevich-local model.
\end{lemma}

\begin{remark}\label{rem: Lk exact triangle}
Likewise, the spectrum $F$ from the above Lemma~\ref{specialreplacement} may be defined by the distinguished triangle
\[
F\to \Sigma^\infty_{S^1} S^0\xrightarrow{0,1}\Sigma^\infty_{S^1} \AA^1.
\]
Let $k\geq 1$ be an integer.
After rotation and smashing with the spectrum $F^{\wedge (k-1)}$, the above triangle becomes $\Sigma^\infty_{S^1} \AA^1 \wedge F^{\wedge (k-1)}[-1]\to F^{\wedge k}\to F^{\wedge (k-1)}$.
Applying $\hhom(-,L^s(E))$ yields the distinguished triangle 
\[
L^{k-1}(E) \to L^k(E) \to \hhom(\Sigma^\infty_{S^1}\AA^1,L^{k-1}(E)[1]). 
\]
Here, $L^{k-1}(E)[1]\simeq L^{k-1}(E[1])$ holds by definition and homotopy-exactness of $L^\simpl$.
\end{remark}

\subsection*{Basechange}
We briefly recall the construction of base change functors in $\AA^1$-homotopy theory.
For details consider the monograph \cite{Ayoub07} and \cite{Hu01}.\smallskip

Let $f\colon R\to S$ be a morphism between noetherian schemes of finite Krull-dimension. There is an adjunction
\[
f^*\colon \sPre_\pt(S)\adjoint \sPre_\pt(R)\colon f_*
\]
where the \emph{direct image} functor $f_*$ is defined by $(f_*G)(-):=G(-\times_S R)$ and where the left adjoint \emph{inverse image} is determined by $f^*(U_\pt):=(U\times_S R)_\pt$ for $U\in\Sm_S$.
The functor $f^*$ is strong symmetric monoidal with respect to the smash product and there is a natural isomorphism $f_*\hhom_\pt(f^*F,G)\cong \hhom_\pt(F,f_*G)$~\cite[(3.4)]{FHM03}.
If the morphism $f\colon R\to S$ is smooth and of finite type, the inverse image functor has a left adjoint
\[
f_\sharp\colon \sPre_\pt(R)\adjoint \sPre_\pt(S)\colon f^*
\]
determined by $f_\sharp(V_\pt)=f_\sharp(V_\pt\to R):=(V\sqcup S\to R\sqcup S\to S)=f_\sharp^{\text{\tiny unpointed}}(V)_\pt$.
We emphasize that throughout the whole text, the functor $f_\sharp$ is considered in this pointed sense: It does not only post-compose with $f$ but also quotients out the base-point along $f$.
In the case of a smooth $f$ of finite type, the inverse image is given by $f^*F=F\wedge R_\pt$, one has a \emph{projection formula} $f_\sharp(G\wedge f^*F)\cong f_\sharp G\wedge F$ (\cf~\eg~\cite[Sec.~5.1]{Hoyois17}) and a natural isomorphism 
\begin{equation}\label{fupperstarandhom}
f^*\hhom_\pt(A,B)\cong \hhom_\pt(f^*A,f^*B)
\end{equation}
by~\cite[Prop.~4.1]{FHM03}.
Note that, since S is noetherian, any open immersion $R\hookrightarrow S$ is smooth and of finite type.

The adjunction $(f^*,f_*)$ is a Quillen adjunction for the objectwise, the Nisnevich-local and the $\AA^1$-Nisnevich-local model structures.
If $f\colon R\to S$ is smooth and of finite type, then the adjunction $(f_\sharp, f^*)$ is a Quillen adjunction for the objectwise, the Nisnevich-local and the $\AA^1$-Nisnevich-local model structures as well and $f^*$ preserves all weak equivalences (\cf\ \cite[Thm.~4.5.10]{Ayoub07}).

\begin{remark}\label{modelstructuresremark}
For the projective versions of the model structures, it is easy to see that $f^*$ and $f_\sharp$ preserve the generating cofibrations and hence all cofibrations.
By the same reason, the objectwise acyclic cofibrations are preserved, so $(f_\sharp, f^*)$ and $(f^*,f_*)$ are Quillen adjunctions for the objectwise structures.
In order to see that the right adjoints $f_*$ and $f^*$ preserve fibrations for the Nisnevich-local and the $\AA^1$-Nisnevich-local model structure, it suffices to show that they preserve fibrations between fibrant objects \cite[Cor.~A.2]{Dugger01}.
As the right adjoints preserve objectwise fibrations, it suffices to show that they preserve fibrant objects.
The fibrant objects of a Bousfield localization may be detected by a particular set $J'$ of acyclic cofibrations \cite[Lem.~3.3.11]{Hirschhorn03}.
It remains to be shown that the left adjoints preserve these acyclic cofibrations in $J'$ which is straightforward (\cf\ \cite[Def.~2.14]{DRO03b}).
\end{remark}

In particular, Remark \ref{modelstructuresremark} implies the following Lemma.

\begin{lemma}\label{interchangebysmooth}
Let $f\colon R\to S$ be a smooth morphism of finite type. For each $F\in \sPre_\pt(S)$, there are canonical (objectwise) weak equivalences 
 \[
 L^s(f^* F)\sim f^*(L^s F) ~\text{ and }~ L^\Aeins(f^* F)\sim f^*(L^\Aeins F)
 \]
in $\sPre_\pt(R)$.
\end{lemma}

The spectrum $S_s^h:=\Spec(\mathcal{O}_{S,s}^h)$ of a henselian local ring of a point $s\in S$ is usually not of finite type over $S$.
Hence, Lemma~\ref{interchangebysmooth} does not apply directly to the canonical morphism $\mathfrak{s}\colon S_s^h\to S$.
Instead, we treat $S_s^h$ as a cofiltered limit of the diagram $\underline D$ given by the affine Nisnevich neighbourhoods of $s$ in $S$ and invoke the following Lemma.

\begin{lemma}\label{limitbasechange}
Let $d\colon D\to S$ be a noetherian $S$-scheme of finite Krull-dimension.
Suppose $d$ is the limit of a cofiltered diagram $\underline D\colon\I\to \Sm_S$ with affine transition morphisms, where $d_i\colon D_i\to S$ denotes the structure morphism of each $D_i:=\underline D(i)$.
Assume that each $D_i$ is quasi-separated.
Let $V \to D$ be an element of $\Sm_D$.
Then, the following statements hold:
\begin{enumerate}
 \item\label{limitbasechange1} There is a cofinal functor $\I_V \to \I$, a cofiltered diagram $\underline{V}\colon\I_V\to \Sm_S$ with affine transition morphisms and a natural transformation $\underline{V} \to \underline{D}_{\mid \I_V}$ inducing $V \to D$ on the limit over $\I_V$ in $\Sch_S$.
 \item\label{limitbasechange2} For $V\to \underline{V}$ as in \eqref{limitbasechange1} and each $F\in \sPre_\pt(S)$, the morphism $\Gamma(\underline{V},d^*F) \to \Gamma(V,d^*F)$ of diagrams induces a canonical natural isomorphism
  \[
   \Gamma(V,d^*F) \cong \underset{i\in {\I_V}}{\colim}\hspace{2pt}\Gamma(V_i,d_i^*F).
  \]
 \item\label{limitbasechange3} For $V\to \underline{V}$ as in \eqref{limitbasechange1} and each $F\in \sPre_\pt(S)$, there is a canonical natural isomorphism of pointed (objectwise) homotopy classes
 \[
    [V_\pt , d^*F]\cong \underset{i\in {\I_V}}{\colim}\hspace{2pt} [V_{i\pt}, d_i^* F] 
 \]  
\item\label{limitbasechange4} In \eqref{limitbasechange1}, open embeddings, \'{e}tale morphisms, smooth morphisms \resp\ Nisnevich-distinguished squares in $\Sm_D$ can be approximated by sectionwise open embeddings, \'{e}tale morphisms, smooth morphisms \resp\ Nisnevich-distinguished squares in $\Sm_S$.
 \item\label{limitbasechange5} For each $F\in \sPre_\pt(S)$, there are canonical (objectwise) weak equivalences 
 \[
 L^s(d^* F)\sim d^*(L^s F) ~\text{ and }~ L^\Aeins(d^* F)\sim d^*(L^\Aeins F)
 \]
in $\sPre_\pt(D)$.
\end{enumerate}
\end{lemma}

\begin{proof}
\eqref{limitbasechange1} Follows from \cite[Thm.~8.8.2, Prop.~17.7.8]{EGA4}.
In fact, we may (and always will) even assume that $\I_V = \I \downarrow i_0$ for a suitable object $i_0 \in \I$ and $\underline{V} = V_{i_0} \times_{D_{i_0}} \underline{D}\vert_{\I \downarrow i_0}$ for a suitable smooth morphism $V_{i_0} \to D_{i_0}$.

For \eqref{limitbasechange2}, we may assume $F$ to be simplicially discrete, \ie, a presheaf.
As we may write $F$ as the colimit over representable presheaves and pullback- as well as section-functors preserve colimits in the category of presheaves, we may assume that $F$ is representable by a suitable object $U \to S$ in $\Sm_S$.
Then $d^*F = U\times_S D$ \resp\ $d_i^*F=U\times_S D_i$ and \eqref{limitbasechange2} follows from \eqref{limitbasechange1} and \cite[Thm.~8.8.2]{EGA4}.

For \eqref{limitbasechange3}, let us first observe that $d^*$ preserves objectwise fibrant objects. Indeed, this holds for the $d_i^*$ by Remark \ref{modelstructuresremark}. Taking sections and applying \eqref{limitbasechange2}, it suffices to observe that a filtered colimit of fibrant simplicial sets is again fibrant.
The assertion of \eqref{limitbasechange3} now follows by taking homotopies with respect to the functorial standard cylinder $(-)\times\Delta^1$.

For \eqref{limitbasechange4}, let $f\colon V^\prime \to V$ be an open embedding (\resp\ an \'{e}tale or smooth morphism) in $\Sm_D$.
We apply \eqref{limitbasechange1} first to the structural map $V \to D$ of the target and then to $V^\prime \to V$, itself.
We get approximations $\underline{V},\underline{V}^\prime \colon \I_f \to \Sm_S$ and a natural transformation
$\underline{f}\colon \underline{V}^\prime \to \underline{V}$ inducing $f$ after taking limits.
By \cite[Prop.~8.6.3]{EGA4} (\resp\ \cite[Prop.~17.7.8]{EGA4}) we may assume that $\underline{f}$ is sectionwise an open embedding (\resp\ an \'{e}tale or smooth morphism).
As in \eqref{limitbasechange1}, we may assume $\underline{f}=f_{i_0}\times_{D_{i_0}}\underline{D}|_{\I\downarrow i_0}$.
\\
Let $f$ be \'{e}tale and $j\colon U \hookrightarrow V$ an open immersion inducing a Nisnevich-distinguished square.
Choose an approximation $\underline{f}=f_{i_0}\times_{D_{i_0}}\underline{D}|_{\I\downarrow i_0}$ of $f$ as above.
By possibly enlarging $i_0$, we can find an approximation $\underline{j}=j_{i_0}\times_{D_{i_0}}\underline{D}|_{\I\downarrow i_0}$ of $j$ by open immersions.
We get a levelwise pullback square

\begin{equation*}
 \xymatrix{
  \underline{U}\times_{\underline{V}} \underline{V}^\prime \ar[r] \ar[d] &
  \underline{V}^\prime \ar[d]^-{\underline{f}}\phantom{.}
 \\
  \underline{U} \ar[r]^-{\underline{j}} &
  \underline{V}.
 }
\end{equation*}
In particular, the sectionwise definition of $\underline{f}^*(\underline{Z}) \to \underline{Z} := \underline{V} \setminus \underline{U}$ (sectionwise with the reduced structure) gives a well defined approximation of $f^*(Z) \to Z := V \setminus U$.
By \cite[Cor.~8.8.2.4]{EGA4} we may even assume that this approximation is sectionwise an isomorphism, \ie, the above square of approximations is sectionwise a Nisnevich-distinguished square. 

For \eqref{limitbasechange5}, note that the first assertion is equivalent to $d^*$ preserving Nisnevich-local fibrant objects and that the second assertion is equivalent to $d^*$ preserving $\AA^1$-Nis\-nevich-local fibrant objects.
Let $F\in\sPre_\pt$ be Nisnevich-local fibrant.
We have to show that $d^*F$ sends Nisnevich-distinguished squares to homotopy pullback squares of simplicial sets.
Let $Q$ be a Nisnevich-distinguished square in $\Sm_D$.
By \eqref{limitbasechange4}, $Q$ may be approximated by a diagram $\underline{Q}$ of Nisnevich-distinguished squares.
By \eqref{limitbasechange2}, we have $(d^*F)(Q)\cong \colim(d_i^*F)(Q_i)$.
Again, as the $d_i^*$ admit Quillen left adjoints for the Nisnevich-local model, it suffices to show that a filtered colimit of homotopy pullback squares of simplicial sets is again a homotopy pullback square.
This, in turn, follows from the fact that those colimits preserve categorical pullback squares, fibrations and weak equivalences of simplicial sets.
This shows $L^s(d^* F)\sim d^*(L^s F)$.
\\
For the second assertion it suffices to show that $d^*$ preserves $\AA^1$-invariant simplicial presheaves.
This is the case for the $d_i^*$ as they admit left adjoints $d_{i,\sharp}$.
The assertion follows directly from \eqref{limitbasechange2}.
\end{proof}

We need the following glueing property. Let $S$ be a base scheme of finite Krull dimension and $i\colon Z\hookrightarrow S$ a closed subscheme with complementary open immersion $j\colon U\hookrightarrow S$. For a pointed simplicial presheaf $F\in \sPre_\pt(S)$, there is a homotopy cofibre sequence
\begin{equation}\label{glueing}
j_\sharp j^* F \to F \to i_*L_\Aeins i^* F
\end{equation}
for the \emph{pointed} $\AA^1$-Nisnevich-local model structure.
This fact follows (\eg~by~\cite[Sec.~5.1]{Hoyois17}) from the unpointed analogue due to Morel and Voevodsky~\cite[Thm.~3.2.21]{MV99} (\cf~also\cite[Thm.~4.5.36]{Ayoub06}).

\medskip

For a morphism $f\colon R\to S$ of noetherian schemes of finite Krull dimension, there is also an adjunction
\[
f^*\colon \Spt_{S^1}(S)\adjoint \Spt_{S^1}(R)\colon f_*
\]
on the level of spectra where one defines $f^*(E)_n:=f^*(E_n)$ and $f_*(D)_n:=f_*(D_n)$ with obvious structure maps.
If the morphism $f\colon R\to S$ is smooth and of finite type, there is an adjunction
\[
f_\sharp\colon \Spt_{S^1}(R)\adjoint \Spt_{S^1}(S)\colon f^*
\]
with $f_\sharp(D)_n:=f_\sharp(D_n)$ and structure maps given by the projection formula.
The adjunction $(f^*,f_*)$ is a Quillen adjunction for the stable objectwise, the stable Nisnevich-local and the stable $\AA^1$-Nisnevich-local model structures, respectively.
If $f\colon R\to S$ is smooth and of finite type, then the adjunction $(f_\sharp, f^*)$ is a Quillen adjunction for these stable model structures as well and $f^*$ preserves all stable weak equivalences (\cf\ \cite[Thm.~4.5.23]{Ayoub07}).

We have the following analogue of Lemma~\ref{interchangebysmooth} and Lemma~\ref{limitbasechange} in the stable setting.

\begin{lemma}\label{interchangebysmoothstable}
Let $\underline{D}$ and $V\rightarrow D$ be as in Lemma~\ref{limitbasechange}.
Let $f\colon R\to S$ be either smooth of finite type or the canonical map $\lim \underline{D} \to S$.
Let $E\in \Spt_{S^1}(S)$ be a spectrum.
Then the following statements hold:
\begin{enumerate}
 \item\label{stablelimitbasechange1} $L^\ob(f^* E)\sim f^*(L^\ob E)$, $L^s(f^* E)\sim f^*(L^s E)$ and $L^\Aeins(f^* E)\sim f^*(L^\Aeins E)$.
 \item\label{stablelimitbasechange2} With the notations of Lemma~\ref{limitbasechange}.\eqref{limitbasechange1}, there is a canonical natural isomorphism of pointed (stable objectwise) homotopy classes
 \[
    [\Sigma_{S^1}^\infty(V_\pt) , d^*E]\cong \underset{i\in {\I_V}}{\colim}\hspace{2pt} [\Sigma_{S^1}^\infty(V_{i\pt}), d_i^* E]. 
 \]  
\end{enumerate}
\end{lemma}
\begin{proof}
The assertions follow from \eqref{fupperstarandhom} and the explicit form of a fibrant replacement~\eqref{taufunctor} using Lemma~\ref{interchangebysmooth} and Lemma~\ref{limitbasechange}.
\end{proof}

\begin{corollary}\label{localcorollary}
Let $E\in \Spt_{S^1}(S)$ be a spectrum. 
Then the following statements are equivalent:
\begin{enumerate}
 \item\label{localcorollary1} The homotopy sheaf $\pi_0(E)$ is trivial.
 \item\label{localcorollary2} For all schemes $V\in\Sm_S$ with structure morphism $p\colon V\to S$ and all points $v\in V$ with canonical morphism $\mathfrak{v}\colon V_v^h:=\Spec(\mathcal{O}_{V,v}^h)\to V$, the homotopy sheaf $\pi_0(\mathfrak{v}^*p^*E)$ is trivial.
 \item\label{localcorollary3} For all points $s\in S$ with canonical morphism $\mathfrak{s}\colon S_s^h\to S$, the homotopy sheaf $\pi_0(\mathfrak{s}^*E)$ is trivial.
\end{enumerate}
\end{corollary}
\begin{proof}
First suppose \eqref{localcorollary2} holds. We want to show \eqref{localcorollary1}, \ie, we have to show that the Nisnevich-stalk at $(V,v)$ of the sheaf $\pi_0(E)$ is trivial for all such $(V,v)$.
By~\eqref{stablelimitbasechange2} of the previous Lemma~\ref{interchangebysmoothstable}, we get
\[
\pi_0(E)_{(V,v)} = \underset{f\colon (W,w)\to(V,v)}{\colim} [\Sigma_{S^1}^\infty(W_\pt), f^*p^*E]\cong [\Sigma_{S^1}^\infty(V_{v,{\pt}}^h), \mathfrak{v}^*p^*E]
\]
where the colimit runs over the Nisnevich-neighbourhoods of $(V,v)$. The identity ${\rm id}_{(V_v^h,v)}$ is cofinal in the Nisnevich-neighbourhoods of $(V_v^h,v)$, so we obtain $[\Sigma_{S^1}^\infty(V_{v,\pt}^h), \mathfrak{v}^*p^*E] = \pi_0(\mathfrak{v}^*p^*E)_{(V_v^h,v)}$ which is trivial by assumption.
\\
For implication \eqref{localcorollary1} $\Rightarrow$ \eqref{localcorollary2}, suppose that $\pi_0(E)=0$.
Let $V_i\to V$ be the diagram given by the affine Nisnevich neighbourhoods of $(V,v)$.
In particular, we have $\lim_i V_i\cong V_v^h$.
By Lemma~\ref{limitbasechange}.\eqref{limitbasechange4}, every object in $\Sm_{V_v^h}$ has the form $W^h_{v}:=W\times_{V}V_v^h$ for a suitable $W\in\Sm_V$.
Take a point $w\in W^h_{v}$ and let $w_0$ be its image in $W$.
Again by Lemma~\ref{limitbasechange}.\eqref{limitbasechange4}, we find a diagram of \'etale maps $W_j\to W$ such that $(W^h_v)^h_w\cong \lim_j (W_j)_v^h\cong \lim_j\lim_i W_j\times_V V_i$.
Using Lemma~\ref{interchangebysmoothstable}.\eqref{stablelimitbasechange2}, we compute
\[
\pi_0(\mathfrak{v}^*p^*E)_{(W^h_v,w)} \cong \colim_j [\Sigma_{S^1}^\infty((W_j)_{v,\pt}^h),\mathfrak{v}^*p^*E]
				      \cong \colim_j \colim_i [\Sigma_{S^1}^\infty((W_j\times_V V_i)_\pt) , p^*E].
\]
The pro-object $\{W_j\times_V V_i\}_{i,j}$ in $\Sm_S$ induces a Nisnevich point $\alpha$ of $\mathrm{Sh}(Sm_S)$.
Note that this point may not correspond to the henselian scheme $X^h_{x}$ for some $X\in\Sm_S$ and $x\in X$ but to a sub-extension of the strict henselization $W_{\bar{w}_0}^{sh}/W_{w_0}^h$.
The assumption $\pi_0(E)=0$ now implies
\[
 \colim_j \colim_i [\Sigma_{S^1}^\infty((W_j\times_V V_i)_\pt) , p^*E]\cong \alpha(\pi_0(E))=0.
\]
As a special case we get implication \eqref{localcorollary3} $\Rightarrow$ \eqref{localcorollary2}.
Finally, the reverse implication \eqref{localcorollary2} $\Rightarrow$ \eqref{localcorollary3} is trivial.
\end{proof}

\subsection*{$\PP^1$-spectra}
In this subsection, we will briefly recall a model for the $\PP^1$-stable \emph{motivic homotopy category}.
As an underlying category of this model structure, we use $(\GG_m,S^1)$-bispectra, \ie, the category $\Spt_{\GG_m}(\Spt_{S^1}(S))$ of $\GG_m$-spectra with entries in $\Spt_{S^1}$ (\cf~\cite[Def.~1.1]{Hovey01}).
Here, by abuse of notation, $\GG_m$ denotes the \mbox{$S^1$-suspension} spectrum of a cofibrant replacement of the pointed object $(\GG_m,1)$ of $\sPre_\pt(S)$ (\cf\ \cite[Rem.~5.1.10]{Morel04}).
Again, by an abuse of notation, we abbreviate this category by $\Spt_{\PP^1}(S)$ and call its objects \emph{$\PP^1$-spectra}.
Similarly to the passage from $\sPre_\pt$ to $S^1$-spectra, the zeroth entry of a $\PP^1$-spectrum and the $\GG_m$-suspension spectrum functor fit into an adjunction
\begin{equation}\label{gmsuspensionadjunction}
\Sigma^\infty_{\GG_m}:\Spt_{S^1}\adjoint\Spt_{\PP^1} :(-)_0.
\end{equation}
For $q\geq 0$, there is also an adjunction $\langle -q\rangle :\Spt_{\PP^1}\adjoint \Spt_{\PP^1}:\langle q\rangle$ of shift functors defined for an integer~$q$ by $E\langle q\rangle_m:=E_{q+m}$ whenever $q+m\geq0$ and $E\langle q\rangle_m:=*$ otherwise.
Again by the general procedure of \cite{Hovey01}, we equip $\Spt_{\PP^1}$ with the stable model structure induced via \eqref{gmsuspensionadjunction} by the stable $\AA^1$-Nisnevich-local structure on $\Spt_{S^1}$ (\cf~\cite[Def.~3.3]{Hovey01}).
Its homotopy category $\SH(S)$ is the (\mbox{$\PP^1$-stable}) \emph{motivic homotopy category}.
The two above-mentioned adjunctions turn into Quillen adjunctions for these structures, respectively.

The motivic homotopy category $\SH$ is a triangulated category with distinguished triangles again given by the homotopy cofibre sequences.
Note that here the triangulated shift is again induced by the simplicial shift $[1]$ and not by the $\GG_m$-shift $\langle 1\rangle$.

Finally, let us mention a concrete fibrant replacement functor for the above model structure on $\Spt_{\PP^1}$. This is completely analogous to the $S^1$-stabilization process from $\sPre_\pt$ to $\Spt_{S^1}$. Let $E\in\Spt_{\PP^1}$. By \cite[Thm.~4.12]{Hovey01}, we may use the functor
\[
\Theta_{\GG_m} E=\colim(E\to \Omega_{\GG_m} E \langle 1\rangle \to (\Omega_{\GG_m})^2 E \langle 2\rangle \to\ldots)
\]
if each level of $E$ is already a fibrant spectrum in $\Spt_{S^1}$. Otherwise we can first apply the stable $\AA^1$-Nisnevich-local fibrant replacement functor $\Theta_{S^1}^\Aeins$ levelwise. We write $\Omega^\infty_{\GG_m}\colon \Spt_{\PP^1}\to \Spt_{S^1}$ for the composition of this fibrant replacement functor with $(-)_0$ from \eqref{gmsuspensionadjunction}.

\subsection*{Preliminaries on t-structures}
We briefly recall the definition of a homological t-structure and basic properties.
Details can be found in \cite{GM03}.

\begin{defi}
A \emph{(homological) t-structure} on a triangulated category $\D$ is a pair of full subcategories $\D_{\leq 0}$ and $\D_{\geq 0}$ which are closed under isomorphisms in $\D$ such that the following axioms hold, where for an integer $n$, one sets $\D_{\geq n}:=\D_{\geq 0}[n]$ and $\D_{\leq n}:=\D_{\leq 0}[n]$:
\begin{enumerate}
 \item\label{item:tstructure1} For all $X\in \D_{\geq 0}$ and all $Y\in \D_{\leq -1}$ we have $\hom_{\D}(X,Y)=0$. 
 \item\label{item:tstructure2} $\D_{\geq 0}$ is closed under $[1]$ (\ie, $\D_{\geq 1}\subseteq \D_{\geq 0}$) and dually $\D_{\leq -1}\subseteq\D_{\leq 0}$.
 \item\label{item:tstructure3} For all $Y\in \D$ there exists a distinguished $X\to Y\to Z\to X[1]$ with $X\in \D_{\geq 0}$ and $Z\in\D_{\leq -1}$.
\end{enumerate}
Set $\D_{=n}:=\D_{\geq n}\cap \D_{\leq n}$ and call $\D_{=0}$ the \emph{heart} of the t-structure.
A t-structure is called \emph{non-degenerate} if $\cap_{n\geq 0}\D_{\geq n}=\{0\}$ and $\cap_{n\leq 0}\D_{\leq n}=\{0\}$.
A t-structure is called \emph{left complete} if for all $X\in \D$ the canonical morphism 
\[
X\to \underset{n\to\infty}{\holim} X_{\leq n}
\]
is an isomorphism.
Dually, a t-structure is called \emph{right complete} if for all $X\in \D$ the canonical morphism $\underset{n\to-\infty}{\hocolim} X_{\geq n}\to X$ is an isomorphism.
\end{defi}

\begin{remark}
The adjunctions
\[
 \begin{array}{rlcll}
  \textit{inclusion}:            &\D_{\geq n}&\rightleftarrows&\D          &:\tau_{\geq n}\\
 \tau_{\leq n}:		     &\D         &\rightleftarrows&\D_{\leq n} &:\textit{inclusion}
 \end{array}
\]
turn $\D_{\geq n}$ into a coreflective and $\D_{\leq n}$ into a reflective subcategory of $\D$.
The counit of the first adjunction is denoted by $(-)_{\geq n}\colon\D\to \D$ and called the \emph{$n$-skeleton}.
The unit of the second adjunction is denoted by $(-)_{\leq n}$ and called the \emph{$n$-coskeleton}.
The skeleton and the coskeleton induce a distinguished triangle
\[
 X_{\geq n}\to X\to X_{\leq n-1}\to (X_{\geq n})[1].
\]
\end{remark}

\begin{remark}\label{completenessanddegeneration}
Let $\D$ be a triangulated category obtained from the homotopy category of a stable model category together with a t-structure.
If the t-structure is left complete, then $\cap_{n\geq 0}\D_{\geq n}=\{0\}$ which can be seen as follows:
Take $X\in \cap_{n\geq 0}\D_{\geq n}$ and suppose that $X\to \holim X_{\leq n}$ is an isomorphism. The homotopy limit of the diagram
\[
\xymatrix{
\ar@{-->}[d]		&\ar@{-->}[d]		&\ar@{-->}[d]	\\
X_{\geq n+1}\ar@{->}[d]\ar@{->}[r]^(0.55)\cong		&X\ar@{=}[d]\ar@{->}[r]		& X_{\leq n}\ar@{->}[d]	\\
X_{\geq n}\ar@{->}[r]^(0.55)\cong		&X\ar@{->}[r]			& X_{\leq n-1}
}
\]
of triangles is the triangle $\holim X_{\geq n}\to X\to \holim X_{\leq n}$.
Since the homotopy limit of weak equivalences is a weak equivalence, the first morphism $\holim X_{\geq n}\to X$ of this triangle is an isomorphism.
This implies $\holim X_{\leq n}\cong 0$ and hence $X\cong 0$.
In the same way, right completeness implies $\cap_{n\leq 0}\D_{\leq n}=\{0\}$.

For the converse, consider \cite[Prop.~1.2.1.19]{Lurie14}:
Suppose that $\D_{\geq 0}$ is stable under countable homotopy products, then $\cap_{n\geq 0}\D_{\geq n}=\{0\}$ implies left complete\-ness.
Dually, if $\D_{\leq 0}$ is stable under countable homotopy coproducts, the relation $\cap_{n\leq 0}\D_{\leq n}=\{0\}$ implies right completeness.
\end{remark}

\begin{proposition}[Ayoub {\cite[Prop.~2.1.70]{Ayoub07}}]\label{tstructureconstruction}
Let $\D$ be a triangulated category with coproducts and let $\S$ be a set of compact objects of $\D$.
Define 
\begin{itemize}
 \item $\D_{\leq -1}$ as the full subcategory of those $Y$ of $\D$ with $\hom_\D(S[n],Y)=0$ for all $n\geq 0$ and all $S\in\S$,
 \item $\D_{\geq 0}$ as the full subcategory of those $X$ of $\D$ with $\hom_\D(X,Y)=0$ for all $Y\in\D_{\leq -1}$.
\end{itemize}
The pair $\D_{\leq 0}=\D_{\leq -1}[1]$ and $\D_{\geq 0}$ forms a t-structure.
The category $\D_{\geq 0}$ is the full subcategory of $\D$ generated under extensions, (small) sums and cones from $\S$ and in particular $\S\subseteq \D_{\geq 0}$. Moreover, the truncation functor $\tau_{\leq -1}$ is given by $\tau_{\leq -1}(X):=\underset{k\to\infty}{\hocolim}~\Phi^k(X)$ where $\Phi(X)$ is defined as the cone
\[
\coprod_{\substack{\operatorname{Hom}(S[n],X) \\ S\in\mathcal{S},n\geq 0 }} S[n] \to X\to \Phi(X).
\]
\end{proposition}

\begin{remark}\label{remarkrightcompleteness}
Let $\D$ be a triangulated category obtained from the homotopy category of a stable model category and let $\S$ be a set of compact objects of $\D$.
The t-structure obtained from the previous Proposition \ref{tstructureconstruction} satisfies the property that $\D_{\leq 0}$ is stable under countable homotopy coproducts.
If $\D$ has an underlying cofibrantly generated model category and $\S$ equals (up to shifts) the set of cofibres of the generating cofibrations, then $\cap_{n\leq 0}\D_{\leq n}=\{0\}$ by \cite[Thm.~7.3.1]{Hovey99} and $\D$ is right complete by the previous Remark \ref{completenessanddegeneration}.
It is however usually a non-trivial issue to show left completeness of a t-structure obtained from Proposition~\ref{tstructureconstruction}.
\end{remark}

\subsection*{t-structures on $S^1$-spectra}
In this subsection we recall some basic properties about canonical t-struc\-tures on $S^1$-spectra arising in $\AA^1$-homotopy theory.

\begin{defi}
Consider the set $\S:=\{\Sigma^\infty_{S^1} U_\pt\mid U\in \Sm_S\}$.
The \emph{objectwise t-structure} (\resp\ \emph{Nisnevich-local t-structure}, \resp\ \emph{$\AA^1$-Nisnevich-local t-structure}) on $\SH_{S^1}^\ob$ (\resp\ $\SH_{S^1}^\simpl$, \resp\ $\SH_{S^1}^\Aeins$) is obtained by applying Proposition~\ref{tstructureconstruction} to the triangulated category $\SH_{S^1}^\ob$ (\resp\ $\SH_{S^1}^\simpl$, \resp\ $\SH_{S^1}^\Aeins$) and to $\S$.
\end{defi}

\begin{remark}
In \cite{Morel05} the Nisnevich-local t-structure on $\SH_{S^1}^\simpl$ is called the \emph{standard t-structure}.
In \cite[Ch.~4.3]{Morel04} the $\AA^1$-Nisnevich-local t-structure on $\SH_{S^1}^\Aeins$ is called the \emph{homotopy t-structure} (on $S^1$-spectra).
\end{remark}

\begin{remark}
By definition we have
\[
\begin{array}{lcl}
\SH^\ob_{S^1\leq -1}\hspace{-5pt} &=& \{Y\in \SH_{S^1}^\ob \hspace{1pt}\mid \text{$Y$ has trivial homotopy presheaves $[\Sigma_{S^1}^\infty(-_\pt)[i], Y]$ for all $i \geq 0$}\}.
\end{array}
\]
Applying the classical \cite[Prop.~3.6]{Margolis83} objectwise, we get
\[
\begin{array}{lcl}
\hspace{9pt}\SH^\ob_{S^1\geq 0}\hspace{2pt} &=& \{X\in \SH_{S^1}^\ob \mid \text{$X$ has trivial homotopy presheaves $[\Sigma_{S^1}^\infty(-_\pt)[i], X]$ for all $i \leq -1$}\}.
\end{array}
\]
The objectwise t-structure is clearly non-degenerate as there are no non-zero spectra without non-trivial homotopy presheaves. The objectwise t-structure is right complete by Remark \ref{remarkrightcompleteness} and left complete by Remark \ref{completenessanddegeneration} as $\SH^\ob_{S^1\geq 0}$ is stable under countable homotopy products.
\end{remark} 

\begin{remark}\label{localizationboundedaboveproposition}
Again, by definition we have
\[
\begin{array}{lcl}
\SH^\simpl_{S^1\leq -1} \hspace{-5pt} &=& \{Y\in \SH_{S^1}^\simpl \hspace{1pt}\mid \text{$Y$ has trivial homotopy presheaves $[\Sigma_{S^1}^\infty(-_\pt)[i], L^\simpl\hspace{2pt} Y]$ for all $i \geq 0$}\}
\end{array}
\]
and using on Nisnevich-stalks the classical result \cite[Prop.~3.6]{Margolis83}, we get
\[
\begin{array}{lcl}
\SH^\simpl_{S^1\geq 0}  \hspace{-5pt}&=& \{X\in \SH_{S^1}^\simpl \hspace{0pt}\mid \text{$X$ has trivial homotopy sheaves $\pi_i^\simpl X$ for all $i \leq -1$}\},\\
\SH^\simpl_{S^1\leq -1} \hspace{-5pt}&=& \{Y\hspace{1.1pt}\in \SH_{S^1}^\simpl \mid \text{$Y$ has trivial homotopy sheaves $\pi_i^\simpl\hspace{1.7pt} Y$ for all $i \geq 0$}\}.
\end{array}
\]
The Nisnevich-local t-structure is clearly non-degenerate as there are no non-zero spectra without non-trivial homotopy sheaves.
The Nisnevich-local t-structure is right complete by Remark \ref{remarkrightcompleteness} and left complete by, \eg, \cite[Lem.~4.4]{Spitzweck14}. 
\end{remark}

\begin{remark}\label{actsbelowremark}
A Nisnevich-local fibrant replacement functor $L^\simpl\hspace{2pt}$ respects only the truncation from above, \ie, if $E$ is in $\SH^\ob_{S^1\leq -1}$, then the spectrum $L^\simpl \hspace{1pt}E$ is in $\SH^\simpl_{S^1\leq -1}$.
The analogous statement is not true for the positive part $\SH^\ob_{S^1\geq 0}$ which can be seen as follows:
By Hilbert's Theorem~90, there is an isomorphism $\Pic(X)\cong H^1_{{\rm Nis}}(X,\GG_{\rm m})$. The Eilenberg--Mac~Lane spectrum $H\GG_{\rm m}$ is in the heart of the objectwise t-structure but
\[
H^1_{{\rm Nis}}(X,\GG_{\rm m}) = [\Sigma^\infty_{S^1} X_\pt,L^s H\GG_{\rm m}[1]] = [\Sigma^\infty_{S^1} X_\pt\hspace{1pt}[-1],L^s H\GG_{\rm m}]
\]
and certainly there are schemes $X$ with non-trivial Picard group.
\end{remark} 

\begin{remark}
By definition and Remark \ref{localizationboundedaboveproposition}, one has
\[
\begin{array}{lcl}
\SH^\Aeins_{S^1\leq -1} \hspace{-5pt}&=& \{Y\in \SH_{S^1}^\Aeins \mid \text{$Y$ has trivial homotopy presheaves $[\Sigma_{S^1}^\infty(-_\pt)[i], L^\Aeins\hspace{2pt} Y]$ for all $i \geq 0$}\}\phantom{.}\\
			\hspace{-5pt}&=& \{Y\in \SH_{S^1}^\Aeins \mid \text{$Y$ has trivial homotopy sheaves $\pi_i^\Aeins\hspace{1.7pt} Y$ for all $i \geq 0$}\}.
\end{array}
\]
The $\AA^1$-Nisnevich-local t-structure is right complete by Remark \ref{remarkrightcompleteness} and we have $\cap_{n\leq 0}\SH^\Aeins_{S^1\leq n}=\{0\}$.
It will be shown in Corollary \ref{tstructurenondegenerate} that the $\AA^1$-Nisnevich-local t-structure is left complete and hence non-degenerate.
\end{remark}

\begin{defi}
We define
\[
\SH^{\Aeins,\pi}_{S^1\geq 0} := \hspace{-5pt}= \{X\in \SH_{S^1}^\Aeins \mid \text{$X$ has trivial homotopy sheaves $\pi_i^\Aeins X$ for all $i \leq -1$}\}.
\]
\end{defi}

\begin{remark}\label{tstructureremark}
The full subcategory $\SH^{\Aeins,\pi}_{S^1 \geq 0}\subseteq \SH_{S^1}^\Aeins$ is closed under homotopy co\-limits and extensions.
There is an inclusion $\SH^{\Aeins,\pi}_{S^1 \geq 0}\subseteq \SH^\Aeins_{S^1\geq 0}$ \cite[Lem.~4.1,~4.3]{Spitzweck14}.
The other implication $\SH^\Aeins_{S^1\geq 0}\subseteq \SH^{\Aeins,\pi}_{S^1 \geq 0}$ holds if and only if $L^\Aeins \Sigma^\infty_{S^1} U_\pt\in \SH^{\Aeins,\pi}_{S^1 \geq 0}$ for all $U\in \Sm_S$.
Unfortunately, there are schemes $S$, such that these two equivalent conditions do not hold (see Remark \ref{counterexample}).
However, they hold true over the spectrum of a field $S$ \cite[Thm.~6.1.8]{Morel05} and we have $\SH^\Aeins_{S^1\geq 0}= \SH^{\Aeins,\pi}_{S^1\geq 0}$ in that case.
\end{remark}

\subsection*{t-structures on $\PP^1$-spectra}

In this subsection we recall the homotopy t-struc\-ture on the motivic homotopy category $\SH$.
We remind the reader that $\langle q\rangle$ denotes the $\GG_m$-shift operation.

\begin{defi}
The \emph{homotopy t-structure} on $\SH$ is obtained by applying Proposition~\ref{tstructureconstruction} to the triangulated category $\SH$ and the set 
\[
\S=\{\Sigma_{\PP^1}^\infty (U_\pt) \langle q \rangle \mid U\in \Sm_S\mbox{ and }q\in\ZZ\}.
\]

\end{defi}

\begin{remark}
We use the name ``homotopy t-structure'' in order to agree with the terminology of \cite{Morel04} and \cite{Morel05}.
\end{remark}

\begin{remark}
Unravelling the definitions, one identifies
\[
\begin{array}{rcl}
\SH_{\leq -1} \hspace{-5pt}&=& \{Y\in \SH \mid \Omega_{\GG_m}^\infty (Y\langle q\rangle) \in \SH_{S^1\leq-1}^\Aeins \text{ for all }q\in\ZZ\}\\
              \hspace{-5pt}&=& \{Y\in \SH \mid (\colim_k \Omega^k_{\GG_m}Y_{k+q})\in \SH_{S^1\leq-1}^\Aeins \text{ for all }q\in\ZZ\}.
\end{array}
\]
In particular
\[
\begin{array}{lcl}
\Omega_{\GG_m}^\infty (\SH_{\leq -1}) & \subseteq & \SH^\Aeins_{S^1 \leq -1} \text{ and}\\
\Sigma_{\GG_m}^\infty (\SH^\Aeins_{S^1 \geq 0}) &\subseteq & \SH_{\geq 0},
\end{array}
\]
using \cite[Lem.~2.1.16]{Ayoub06} for the latter.
\end{remark}

\begin{remark}
Over a field, using \cite[Lem.~4.3.11]{Morel04} and the equality $\SH^\Aeins_{S^1\geq 0}= \SH^{\Aeins,\pi}_{S^1\geq 0}$ from Remark \ref{tstructureremark}, we can also identify
\[
\SH_{\geq 0}  = \{X\in \SH \mid \Omega_{\GG_m}^\infty(X\langle q\rangle)\in \SH^\Aeins_{S^1 \geq 0} \text{ for all }q\in\ZZ\}
\]
(\cf\ \cite[Ch.~5.2]{Morel04}). In particular, we have $\Omega_{\GG_m}^\infty (\SH_{\geq 0})  \subseteq  \SH^\Aeins_{S^1 \geq 0}$ in this case.
\end{remark}

\begin{remark}
\label{rightcompletenessofhomotopytstructure}
The homotopy t-structure on the motivic homotopy category is right complete by Remark~\ref{remarkrightcompleteness} and we have $\cap_{n\leq 0}\SH_{\leq n}=\{0\}$.
It will be shown in Corollary~\ref{homotopytstructurenondegenerate} that the homotopy t-structure on the motivic homotopy category is also left complete and hence non-degenerate.
\end{remark}

\section{Gabber-presentations over henselian discrete valuation rings}\label{chaptergabber}

\noindent
Throughout the whole section, fix a henselian discrete valuation ring $\mathfrak{o}$ with maximal ideal $\mathfrak{m}\trianglelefteq\mathfrak{o}$, local uniformizer $\pi\in\mathfrak{m}$, residue field $\mathbb{F} = \mathfrak{o}/\mathfrak{m}$ and field of fractions $k$.
Assume that $\mathbb{F}$ is an infinite field.
Let $S$ be the spectrum of $\mathfrak{o}$ and denote by $\sigma$ \resp\ $\eta$ the closed \resp\ generic point of $S$.
We want to prove the following version of Gabber's geometric presentation lemma over $\mathfrak{o}$:

\begin{theorem}\label{thm: gabber distinguished square}
Let $\mathfrak{o}$ be a henselian discrete valuation ring with infinite residue field.
Let ${X}/\mathfrak{o}$ be a smooth $\mathfrak{o}$-scheme of finite type and let ${Z}\hookrightarrow {X}$ be a proper closed subscheme.
Let $z$ be a point in ${Z}$.
If $z$ lies in the special fibre, suppose that ${Z}_\sigma \neq {X}_\sigma$.
Then, Nisnevich-locally around $z$, there exists a smooth $\mathfrak{o}$-scheme $V$ of finite type and a cartesian square 
\[
  \xymatrix{
   {X}\setminus{Z} \ar[r] \ar[d] &
   {X} \ar[d]^-p
  \\
   \mathbb{A}_V^1\setminus p({Z}) \ar[r] &
   \mathbb{A}_V^1
  }
\]
such that $p$ is \'{e}tale, the restriction $p|_Z\colon Z\hookrightarrow \mathbb{A}_V^1$ is a closed subscheme and $Z$ is finite over $V$.
In particular, this square is a Nisnevich-distinguished square and therefore, the induced canonical morphism ${X} / ({X}\setminus{Z}) \rightarrow \mathbb{A}_V^1/(\mathbb{A}_V^1\setminus p({Z}))$ is an isomorphism of Nisnevich sheaves.
\end{theorem}

\begin{remark}\label{rem: earlier results}
The essential case of Theorem~\ref{thm: gabber distinguished square} is that of an effective Cartier-divisor $Z\hookrightarrow X$ (\cf~proof of Theorem~\ref{thm: gabber presentation}, below).
Earlier results for relative effective Cartier-divisors $Z\hookrightarrow X$ over discrete valuation rings are \cite[Lem.~1]{GilletLevine87} and \cite[Thm.~3.4]{Dutta95}.
These results are even Zariski local and do not assume infinite residue fields. However, they do not include an analogue for the crucial finiteness claim of $Z/V$.
\end{remark}

The map $p$ in Theorem~\ref{thm: gabber distinguished square} will be provided by a careful choice of suitable linear projections.
Before we give a short outline of the proof, let us first recall some basic facts about linear projections.

\subsection*{Linear projections}
Denote by $\mathbb{A}_{x_1/x_0,\dots,x_N/x_0,S} = \mathbb{A}_{\underline{x}/x_0,S}$ \resp\ $\mathbb{P}_{x_0:{\dots}:x_N,S} = \mathbb{P}_{\underline{x},S}$ the affine \resp\ projective $N$-space $\mathbb{A}_S^N$ \resp\ $\mathbb{P}_S^N$ with coordinates $\frac{x_1}{x_0},\dots,\frac{x_N}{x_0}$ \resp\ homogeneous coordinates $x_0:{\dots}:x_N$.
We get the standard open embedding $\mathbb{A}_{\underline{x}/x_0,S}\hookrightarrow \mathbb{P}_{\underline{x},S}$.
By abuse of notation, we identify $\frac{x_1}{x_0},\dots,\frac{x_N}{x_0}$ with $x_1,\dots,x_N$ and write just $\mathbb{A}_{\underline{x},S}$ for $\mathbb{A}_{\underline{x}/x_0,S}$ (and similarly for other coordinates).

Let $\mathbb{A}_{x_1,\dots,x_N,S}^\vee = \mathbb{A}_{x_1^\ast,\dots,x_N^\ast,S}$ be the dual affine space, \ie, for any $\mathfrak{o}$-Algebra $A$, $\mathbb{A}_{x_1,\dots,x_N,S}^\vee(A)$ is the free $A$-module generated by the coordinate functions $x_1,\dots,x_N$ of $\mathbb{A}_{\underline{x},S}$.
Dually we can view $\mathbb{A}_{\underline{x},S}(A)$ as the free $A$-module generated by the dual coordinate functions $x_1^\ast,\dots,x_N^\ast$.
To be more precise, take $r$ copies of $\mathbb{A}_{\underline{x},S}^\vee$ and denote the $j^{\rm th}$-copy by $\mathbb{A}_{\underline{x},S}^{\vee,(j)}$ with coordinate functions $x_{i,j}^\ast := x_i^\ast$ for $1\leq i \leq N$.
Mapping $t_j \mapsto \sum_ix_i\otimes x_{i,j}^\ast$ defines the dual pairing $\langle-,-\rangle\colon\mathbb{A}_{\underline{x},S} \times_S \mathbb{A}_{\underline{x},S}^{\vee,(j)} \rightarrow \mathbb{A}_{t_j,S}$.
Via this pairing, each $A$-point $u$ of $\mathbb{A}_{\underline{x},S}^{\vee,(j)}$ induces a linear $A$-morphism (in abuse of notation also denoted by) $u$ defined as the composition
\begin{equation*}
 \xymatrix{
  \mathbb{A}_{\underline{x},A} \ar[r]^-{{\rm id}\times u}&
  \mathbb{A}_{\underline{x},A} \times_A \mathbb{A}_{\underline{x},A}^{\vee,(j)} \ar[r]^-{\langle-,-\rangle} &
  \mathbb{A}_{t_j,A}
 }
 \text{ via }
 \xymatrix{
  t_j \ar@{|->}[r] &
  \sum_i\langle x_i^\ast,u\rangle x_i.
 }
\end{equation*}
This map is precisely the linear form given by the $A$-point $u$ seen as the corresponding linear combination of the $x_i$ in $\mathbb{A}_{\underline{x},S}^{\vee,(j)}(A) = \bigoplus_i Ax_i$.
We define
\begin{equation*}
 \mathfrak{E}_r := \mathbb{A}_{\underline{x},S}^{\vee,(1)} \times_S {\dots} \times_S \mathbb{A}_{\underline{x},S}^{\vee,(r)}
\end{equation*}
and look at it as the space of linear projections $\mathbb{A}_{\underline{x},S} \rightarrow \mathbb{A}_{t_1,\dots,t_r,S}$.
Indeed, each $A$-point $\underline{u}$ of $\mathfrak{E}_r$ induces a linear $A$-morphism
\begin{equation*}
 \underline{u}\colon
 \xymatrix{
  \mathbb{A}_{\underline{x},A} \ar[r] &
  \mathbb{A}_{t_1,\dots,t_r,A}
 }
 \text{ via }
 \xymatrix{
  t_j \ar@{|->}[r] &
  \sum_i \langle x_{i}^\ast, u_j \rangle x_i.
 }
\end{equation*}
Mapping $t_0\mapsto x_0$, this extends to a rational map
\begin{equation*}
 \underline{u}\colon
 \xymatrix{
  \mathbb{P}_{\underline{x},A} \ar@{-->}[r] &
  \mathbb{P}_{t_0:{\dots}:t_r,{A}}
 }
\end{equation*}
with locus of indeterminacy $L_{\underline{u}} := V_+(x_0,u_1,\dots,u_r) \subseteq H_\infty$, where $u_j$ is the linear form in the coordinates $x_1,\dots,x_N$ corresponding to the $j^{\rm th}$ component of $\underline{u}$ and $H_\infty \subset \mathbb{P}_{\underline{x},A}$ is the hyperplane at infinity $V_+(x_0)$.

Assume ${Y} \hookrightarrow \mathbb{A}_{\underline{x},A}$ is a (reduced) closed subscheme with (reduced) projective closure $\bar{{Y}}\hookrightarrow \mathbb{P}_{\underline{x},A}$ such that $\bar{{Y}} \cap L_{\underline{u}} = \emptyset$.
Then $\underline{u}$ induces regular maps
\begin{equation*}
 p_{\underline{u}}\colon
 \xymatrix{
  {Y} \ar[r] &
  \mathbb{A}_{t_1,\dots,t_r,A}
 }
 \text{ and }
 \bar{p}_{\underline{u}}\colon
 \xymatrix{
  \bar{{Y}} \ar[r] &
  \mathbb{P}_{t_0:{\dots}:t_r,A}
 }
\end{equation*}
satisfying $ p_{\underline{u}} = \bar{p}_{\underline{u}} \times_{\mathbb{P}_{\underline{t},A}}\mathbb{A}_{\underline{t},A}$.
Observe that \cite[Thm.~I.5.3.7]{Shafarevich74} remains true in our setting:

\begin{lemma}\label{lem: shafarevich}
For any $\underline{u} \in \mathfrak{E}_r(A)$ and any closed ${Y} \hookrightarrow \mathbb{A}_{\underline{x},S}$ with $\bar{{Y}} \cap L_{\underline{u}} = \emptyset$, the linear projections $p_{\underline{u}}$ and $\bar{p}_{\underline{u}}$ are finite maps.
\end{lemma}
\begin{proof}
It suffices to show that $\bar{p}_{\underline{u}}$ is finite.
As a map between projective schemes over $S$, $\bar{p}_{\underline{u}}$ itself is projective. 
It remains to show that $\bar{p}_{\underline{u}}$ is quasi-finite:
Let $\bar{\sigma}$ \resp\ $\bar{\eta}$ be a geometric point of $S$ over $\sigma$ \resp\ $\eta$. 
By \cite[Thm.~I.5.3.7]{Shafarevich74}, $\bar{\sigma}^*\bar{p}_{\underline{u}}$ \resp\ $\bar{\eta}^*\bar{p}_{\underline{u}}$ is finite.
It follows that $\bar{p}_{\underline{u}}$ is finite on the special \resp\ geometric fibre ${\sigma}^*\bar{p}_{\underline{u}}$ \resp\ ${\eta}^*\bar{p}_{\underline{u}}$, hence quasi-finite.
\end{proof}

\subsection*{Outline of the proof of Theorem~\ref{thm: gabber distinguished square}}
The proof of Theorem~\ref{thm: gabber distinguished square} will principally follow the proof of Gabber's geometric presentation lemma over fields in \cite{CTHK97}: 
The crucial part of Theorem~\ref{thm: gabber distinguished square} turns out to be the finiteness claim for $Z/V$.
We make the Ansatz $p = p_{\underline{u}}$ for a closed embedding $i\colon X\hookrightarrow \mathbb{A}_{x_1,\dots,x_N,S}$ and a linear projection $\underline{u} \in \mathfrak{E}_n(\mathfrak{o})$, for $n$ the relative dimension of $X/S$.
Using Lemma~\ref{lem: shafarevich}, one can show that the property
\begin{center}
\textit{the induced map $p_{(u_1,\dots,u_{n-1})}\vert_Z\colon Z\rightarrow \mathbb{A}_{t_1,\dots,t_{n-1},S}$ is finite} 
\end{center}
is an open property in our space of linear projections $\mathfrak{E}_n$ (\cf~proof of Lemma~\ref{lem: finitness}, below).
If $W \hookrightarrow \mathfrak{E}_n$ is the corresponding open locus, we first need to make sure that the special fibre $W_\sigma$ is nonempty -- because the residue field $\mathbb{F}$ is infinite and $\mathfrak{E}_n$ isomorphic to an affine space, $W(\mathfrak{o})$ will automatically be nonempty in this case.
We will see in the proof of Lemma~\ref{lem: finitness} that $W_\sigma$ is nonempty, as soon as $X$ over $S$ is fibrewise dense inside its closure in $\mathbb{P}_{x_0:\dots:x_N,S}$.
In fact, we just need this closure to be dense on the special fibre, but this obviously is equivalent to fibrewise density.
Therefore, special care needs to be taken about the choice of our initial closed embedding $i\colon X \hookrightarrow \mathbb{A}_{\underline{x},S}$.
In Proposition~\ref{prop: kay density lemma} we will provide a closed embedding of this type, but the price to pay is that we need to replace $(X,z)$ by a suitable Nisnevich-neighbourhood.

Let $F = F_{\underline{u}}$ denote the set of preimages under $p_{(u_1,\dots,u_{n-1})}$ of $p_{(u_1,\dots,u_{n-1})}(z)$ in $Z$.
The next goal is to find $\underline{u}$ in $W(\mathfrak{o})$ such that $p_{\underline{u}}$ is \'{e}tale around each point of $F$ and the restriction of $p_{\underline{u}}$ to $F$ is universally injective.  
This again corresponds to non-empty open conditions in (the special fibre of) $\mathfrak{E}_n$ (\cf~the proof of Lemma~\ref{lem: etale and radicial}).
Shrinking $W$ accordingly, we may assume that this is the case for all $\underline{u} \in W(\mathfrak{o})$ (\cf~Proposition~\ref{prop: gabber A}).

Making use of the finiteness of $p_{(u_1,\dots,u_{n-1})}\vert_Z$, we get an open neighbourhood $V = V_{\underline{u}}$ of $p_{(u_1,\dots,u_{n-1})}(z)$ in $\mathbb{A}_{t_1,\dots,t_{n-1},S}$ such that $p_{\underline{u}}$ is \'{e}tale at all points of $Z\cap p_{(u_1,\dots,u_{n-1})}^{-1}(V)$ and the restriction of $p_{\underline{u}}$ induces a closed embedding $Z\cap p_{(u_1,\dots,u_{n-1})}^{-1}(V) \hookrightarrow \mathbb{A}_V^1$ (\cf~Lemma~\ref{lem: securing V}).
Finally, replacing $(X,z)$ by a suitable Zariski-neighbourhood, we establish in Lemma~\ref{lem: securing U} the remaining properties claimed in Theorem~\ref{thm: gabber distinguished square}.

Summing up, we will show slightly more in fact:
We will prove the following version of \cite[Thm.~3.2.2]{CTHK97} over $\mathfrak{o}$.

\begin{theorem}\label{thm: gabber presentation}
Let ${X}=\Spec({A})/S$ be a smooth affine $S$-scheme of finite type, fibrewise of pure dimension $n$ and let ${Z} =\Spec({B}) \hookrightarrow {X}$ be a proper closed subscheme.
Let $z$ be a point in ${Z}$.
If $z$ lies in the special fibre, suppose that ${Z}_\sigma \neq {X}_\sigma$.
Then, Nisnevich-locally around $z$, there exists a closed embedding ${X} \hookrightarrow \mathbb{A}_{S}^N$ and a Zariski-open subset $W\subseteq \mathfrak{E}_n$ with $W(\mathfrak{o}) \neq \emptyset$, \st\ the following holds:
\\
For all $\underline{u} \in W(\mathfrak{o})$ with linear projections $p_{\underline{u}} = p_{(u_1,\dots,u_{n-1})} \times_S p_{u_n}\colon {X} \rightarrow \mathbb{A}_S^n = \mathbb{A}_S^{n-1} \times_S \mathbb{A}_S^{1}$ there are Zariski-open neighbourhoods $p_{(u_1,\dots,u_{n-1})}(z)\in V \subseteq \mathbb{A}_S^{n-1}$ and $z\in U \subseteq p_{(u_1,\dots,u_{n-1})}^{-1}(V)$ satisfying:
\begin{enumerate}
 \item\label{maintheoremone} $p_{(u_1,\dots,u_{n-1})}\vert_{{Z}}\colon {Z} \rightarrow \mathbb{A}_S^{n-1}$ is finite,
 \item\label{maintheoremtwo} ${Z} \cap U = {Z}\cap p_{(u_1,\dots,u_{n-1})}^{-1}(V)$,
 \item\label{maintheoremthree} $p_{\underline{u}}\vert_U\colon U\rightarrow \mathbb{A}_S^n$ is \'{e}tale and restricts to a closed embedding ${Z} \cap U \hookrightarrow \mathbb{A}_V^1$ and
 \item\label{maintheoremfour} $p_{\underline{u}}^{-1}(p_{\underline{u}}({Z} \cap U))\cap U = {Z} \cap U$.
\end{enumerate}
\end{theorem}

The proof of Theorem~\ref{thm: gabber presentation} will follow the proof in \cite{CTHK97} \resp~the outline sketched above.

\begin{proof}
Clearly, we may assume that both ${X}$ and ${Z}$ are connected.
Next, observe that the case of $z$ lying in the generic fibre ${X}_\eta$ of ${X}/S$ is already covered by \cite[Thm.~3.2.2]{CTHK97}.
Thus, we may assume that $z$ lies in the special fibre ${X}_\sigma$ of ${X}/S$.
Finally, observe that we may enlarge ${Z}$.
In particular, picking any element $f$ in the kernel of ${A}\twoheadrightarrow {B}$ with $f\neq 0$ in the special fibre ${A}\otimes_{\mathfrak{o}}\mathbb{F}$, we may assume ${B} = {A}/f$, \ie, ${Z}=V(f)$.
\\
We will follow the outline of the proof sketched above:
Up to a refinement by a suitable Nisnevich-neighbourhood $(X^\prime,z^\prime)\rightarrow(X,z)$, Proposition~\ref{prop: kay density lemma} will provide a closed embedding $i_0\colon {X}^\prime \hookrightarrow \mathbb{A}_{x_1,\dots,x_N,S}$ such that $Z^\prime = Z\times_XX^\prime$ is fibrewise dense in its Zariski-closure $\bar{Z}^\prime$ in $\mathbb{P}_{x_0:\dots:X_N,S}$.
Replacing our base-point $z$ by a specialization to a closed point $z_0$ in the image of $X^\prime \rightarrow X$, we can assume that $z$ is closed itself (\cf~Reduction~\ref{red: closed point}).
Further, we replace $(X,z)$ by $(X^\prime,z^\prime)$, \ie, we assume $X^\prime = X$.
Composing the closed embedding $i_0\colon {X} \hookrightarrow \mathbb{A}_{\underline{x},S}$ with a linear projection $\mathbb{A}_{\underline{x},S} \rightarrow \mathbb{A}_{t_1,\dots,t_n,S}$, corresponding to an $\mathfrak{o}$-point $\underline{u}$ of the space of linear projections $\mathfrak{E}_n$, we get maps
\begin{equation*}
 \xymatrix{
  X \ar[r]^-{p_{\underline{u}}} \ar[rd]_-{p_{(u_1,\dots,u_{n-1})}\hspace{2ex}} &
  \mathbb{A}_{t_1,\dots,t_n,S} \ar[d]^-{{\rm pr}} \phantom{.} \\
  &
  \mathbb{A}_{t_1,\dots,t_{n-1},S}.
 }
\end{equation*}
Here $n$ is the dimension of $X$.
Proposition~\ref{prop: gabber A} will provide an open $W$ in our space of linear projections $\mathfrak{E}_n$ with $W(\mathfrak{o})$ non-empty and for each linear projection $\underline{u}$ in $W(\mathfrak{o})$ the restriction $p_{(u_1,\dots,u_{n-1})}\vert_Z$ is finite (\ie, part~\eqref{maintheoremone} in Theorem~\ref{thm: gabber presentation}), $p_{\underline{u}}$ is \'{e}tale around $F = F_{\underline{u}} = p_{(u_1,\dots,u_{n-1})}^{-1} (p_{(u_1,\dots,u_{n-1})}(z))\cap Z$ and $p_{\underline{u}}\vert_F\colon F \rightarrow p_{\underline{u}}(F)$ is universally injective.
\\
Fix any such $\underline{u}$ in $W(\mathfrak{o})$.
In Lemma~\ref{lem: securing V}, we will replace $\mathbb{A}_{t_1,\dots,t_{n-1},S}$ by a Zariski-neighbourhood $V=V_{\underline{u}}$ of $p_{(u_1,\dots,u_{n-1})}(z)$ such that $p_{\underline{u}}$ is \'{e}tale around every point of $Z\cap p_{(u_1,\dots,u_{n-1})}^{-1}(V)$ and such that the induced restriction $p_{\underline{u}}\vert_{Z\cap p_{(u_1,\dots,u_{n-1})}^{-1}(V)}\colon Z\cap p_{(u_1,\dots,u_{n-1})}^{-1}(V) \hookrightarrow \mathbb{A}_{t_n,V}$ is a closed embedding.
In Lemma~\ref{lem: securing U}, we will shrink $p_{(u_1,\dots,u_{n-1})}^{-1}(V)$ to a Zariski-neighbourhood $U_1$ of $z$ satisfying the analogue of \eqref{maintheoremfour} in Theorem~\ref{thm: gabber presentation}, \ie, $p_{\underline{u}}^{-1}(p_{\underline{u}}(Z\cap U_1))\cap U_1 = Z\cap U_1$, without changing $Z\cap p_{(u_1,\dots,u_{n-1})}^{-1}(V) = Z\cap U_1$.
In particular, $p_{\underline{u}}$ restricts to a closed embedding $Z\cap U_1 \rightarrow \mathbb{A}_{t_n,V}$.
Since $p_{\underline{u}}$ is \'{e}tale already around every point of $Z\cap U_1$, we may shrink $U_1$ a bit more (by intersecting it with the open \'{e}tale locus of $p_{\underline{u}}$) to get the desired Zariski-neighbourhood $U=U_{\underline{u}}$ of $(X,z)$ satisfying \eqref{maintheoremtwo}, \eqref{maintheoremthree} and \eqref{maintheoremfour} in Theorem~\ref{thm: gabber presentation}.
\end{proof}

\subsection*{Towards the finiteness part}
The key part in the proof of Theorem~\ref{thm: gabber presentation} is the finiteness assertion \eqref{maintheoremone}.
By Lemma~\ref{lem: shafarevich}, we need to find a closed embedding $i_0\colon {X} \hookrightarrow \mathbb{A}_{\underline{x},S}$ and an $\mathfrak{o}$-point $\underline{u}\in \mathfrak{E}_n(\mathfrak{o})$ \st\ the closure $\bar{{Z}}$ of ${Z}$ in $\mathbb{P}_{\underline{x},S}$ intersects $L_{(u_1,\dots,u_{n-1})}$ trivially.
Unfortunately, it is not enough to require that the fibrewise closure of ${Z}$ in $\mathbb{P}_{\underline{x},S}$ misses $L_{(u_1,\dots,u_{n-1})}$.
Indeed, $\bar{{Z}}$ might not be the fibrewise projective closure of ${Z}$ over $S$ -- the special fibre $\bar{Z}_\sigma$ might be strictly larger than the projective closure of $Z_\sigma$:

\begin{example}\label{ex: not fibrewise dense}
Let ${A} = \mathfrak{o}[x_1]$ and ${B}= \mathfrak{o}[x_1]/(\pi x_1^2 +x_1+1)$.
Then $\Spec({B}) \subset \mathbb{P}_{x_0:x_1,S}$ is fibrewise closed but at least one solution of $\pi x_1^2 +x_1+1$ in $k^{\rm alg}$ specializes to $\infty$ in $\mathbb{P}_{x_0:x_1,\mathbb{F}}$, \ie, $\Spec({B}) \subset \mathbb{P}_{x_0:x_1,S}$ is not closed.
\end{example}

To avoid these difficulties, we need to make a careful choice for the embedding $i_0\colon {X} \hookrightarrow \mathbb{A}_{\underline{x},S}$:

\begin{proposition}\label{prop: kay density lemma}
Nisnevich-locally around $z$, there exists  a closed embedding $i_0\colon {X} \hookrightarrow \mathbb{A}_{\underline{x},S}$ with ${Z}$ fibrewise dense over $S$ inside its closure $\bar{{Z}}$ in $\mathbb{P}_{\underline{x},S}$.
\end{proposition}
\begin{proof}
We need to adapt \cite[Thm.~4.6]{Kai15} to our situation.
Since $z$ lies in the special fibre and $\mathfrak{o}$ is henselian, \loccit\ gives us an affine Nisnevich-neighbourhood $({Z}^\prime,z^\prime) \rightarrow ({Z},z)$ and a closed embedding $\bar{i}_0\colon{Z}^\prime\hookrightarrow \mathbb{A}_{x_1,\dots,x_m,S}$, \st\ ${Z}^\prime$ is fibrewise dense over $S$ inside its closure $\bar{{Z}}^\prime$ in $\mathbb{P}_{x_0:{\dots}:x_m,S}$.
Since both $Z$ and $Z^\prime$ are affine, the underling \'{e}tale morphism $Z^\prime \rightarrow Z$ is standard smooth, \ie, $Z^\prime = {\rm Spec}(B^\prime)$ with ${B}^{\prime} = {B}[t_1,\dots,t_s]/(\bar{g}_1,\dots,\bar{g}_s)$ and invertible Jacobi-determinant ${\rm det}(\{ \frac{\partial\bar{g}_i}{\partial t_j} \}_{i,j}) \in {B}^{{\prime}\times}$. 
\\
We want to extend $({Z}^\prime,z^\prime)$ to a Nisnevich-neighbourhood $({X}^\prime,z^\prime)$ of $({X},z)$:
Since the Jacobi-determinant ${\rm det}(\{ \frac{\partial\bar{g}_i}{\partial t_j} \}_{i,j})$ is invertible in ${B}^{\prime}$, it is non-trivial in ${B}^{\prime} \otimes k(z^\prime)$.
Choose a lift $g_i \in {A}[\underline{t}]$ for each $\bar{g}_i$ and set ${A}^{\prime} := {A}[\underline{t}]/(g_1,\dots,g_s)$ and ${X}^{\prime} = \Spec({A}^{\prime})$.
By construction, ${B}^{\prime} = {A}^{\prime} \otimes_{{A}}{B}$, so $z^\prime$ induces a point (also denoted by) $z^\prime$ in ${X}^{\prime}$.
Since ${\rm det}(\{ \frac{\partial g_i}{\partial t_j} \}_{i,j}) \equiv {\rm det}(\{ \frac{\partial\bar{g}_i}{\partial t_j} \}_{i,j}) \neq 0$ in ${A}^{\prime} \otimes k(z^\prime)$, the Jacobi-determinant ${\rm det}(\{ \frac{\partial g_i}{\partial t_j} \}_{i,j})$ is invertible around $z^\prime$ in ${X}^{\prime}$.
By shrinking ${X}^{\prime}$ without changing ${Z}^{\prime}$ (since the Jacobi-determinant is invertible on the latter), we may assume that $({Z}^{\prime},z^\prime) \rightarrow ({Z},z)$ extends to an affine Nisnevich-neighbourhood $({X}^{\prime},z^\prime) \rightarrow ({X},z)$.
Further, lifting the images of $x_i$ in ${B}^\prime$ to ${A}^\prime$, ${\bar{i}_0}$ extends to a map ${i_0^\prime\colon {X}^\prime \rightarrow \mathbb{A}_{x_1,\dots,x_n,S}}$.
\\
Unfortunately, there is no reason for ${i_0^\prime}$ to be a closed embedding.
To repair this, choose a closed embedding ${X}^\prime \hookrightarrow \mathbb{A}_{y_1,\dots,y_r,S}$ over $S$, \ie, generators $a_j$ of ${A}^\prime$ as an $\mathfrak{o}$-algebra.
Recall that we assumed $B=A/f$ for $f$ in $A$ non-zero in the special fibre $A\otimes_{\mathfrak{o}}\mathbb{F}$ (\cf~proof of Theorem~\ref{thm: gabber presentation}).
Writing $\underline{a}^{\underline{j}} = a_1^{j_1} \cdot {\dots} \cdot a_r^{j_r}$, any element of the ideal $f\cdot {A}^\prime$ is of the form $\sum_{\underline{j}} \lambda_{\underline{j}}\cdot f\underline{a}^{\underline{j}}$ where $\lambda_{\underline{j}} \in \mathfrak{o}$ and $\underline{j}$ runs through a finite subset of $\mathbb{N}^r$.
Mapping $y^{(\underline{j})} \mapsto f\underline{a}^{\underline{j}}$, we get a map ${X}^\prime \rightarrow \mathbb{A}_{\{y^{(\underline{j})} \vert \underline{j}\in \mathbb{N}^r \},S}$ into a copy of the infinite affine space over $S$.
Together with the map $i_0^\prime\colon X^\prime \rightarrow \mathbb{A}_{x_1,\dots,x_n,S}$, we get a \emph{closed} embedding $i_\infty\colon {X}^\prime \hookrightarrow \mathbb{A}_{x_1,\dots,x_m,S} \times_S \mathbb{A}_{\{y^{(\underline{j})} \vert \underline{j}\in \mathbb{N}^r \},S} \cong \mathbb{A}_S^\infty$ into the fibre product:
Indeed, $i_0^\prime\colon \mathfrak{o}[x_1,\dots,x_m] \rightarrow {A}^\prime$ is surjective modulo $f$ and $\mathfrak{o}[y^{(\underline{j})} \vert \underline{j}\in \mathbb{N}^r] \rightarrow {A}^\prime$ has image $\mathfrak{o}[f\cdot{A}^\prime]$ by construction.
By Lemma~\ref{lem: maps to infinite affine spaces}, below, $i_\infty$ induces $i_0\colon {X}^\prime \hookrightarrow \mathbb{A}_{x_1,\dots,x_m,S} \times_S \mathbb{A}_{y^{(\underline{i}_1)},\dots,y^{(\underline{i}_l)},S}$ still a closed embedding for $\underline{j}_1,\dots,\underline{j}_l \in \mathbb{N}^r$ suitable.
Setting $x_{m+s} := y^{(\underline{j}_s)}$ and $N := m+l$, we have constructed a closed embedding $i_0\colon {X}^\prime \hookrightarrow \mathbb{A}_{x_1,\dots,x_N,S}$ such that $i_0\vert_{{Z}^\prime}$ factors over $\bar{i}_0\colon {Z}^\prime \hookrightarrow \mathbb{A}_{x_1,\dots,x_m,S} = V(x_{m+1},\dots,x_N) \subseteq \mathbb{A}_{x_1,\dots,x_N,S}$.
In particular, the closure of ${Z}^\prime$ in $\mathbb{P}_{x_0:{\dots}:x_N,S}$ is just $\bar{{Z}}^\prime$ inside the linear subspace $V_+(x_{m+1},\dots,x_N) \subseteq \mathbb{P}_{x_0:\dots:x_N,S}$, so ${Z}^\prime$ is fibrewise dense over $S$ inside this closure.
\end{proof}

\begin{lemma}\label{lem: maps to infinite affine spaces}
Let ${C}$ be an $\mathfrak{o}$-algebra of finite type.
Let $\iota\colon \Spec({C})\hookrightarrow \mathbb{A}_{t_1,t_2,\dots,S} = \mathbb{A}_S^\infty$ be a closed embedding and let ${\rm pr}_{\leq N}\colon \mathbb{A}_{t_1,t_2,\dots,S} \rightarrow \mathbb{A}_{t_1,\dots,t_N,S} =\mathbb{A}_S^N$ be the canonical projection.
Then ${\rm pr}_{\leq N} \circ \iota$ is a closed embedding for $N\gg 0$.
\end{lemma}
\begin{proof}
Say, as an $\mathfrak{o}$-algebra, ${C}$ is generated by $c_1,\dots,c_r \in {C}$.
Since the corresponding map on algebras $\mathfrak{o}[t_1,t_2,\dots] \twoheadrightarrow {C}$ is surjective, we can find polynomials $f_i \in \mathfrak{o}[t_1,t_2,\dots]$ mapping to $c_i$.
Pick $N\gg 0$ \st~all the $f_i$ lie inside $\mathfrak{o}[t_1,\dots,t_N]$.
Then $\iota$ restricted to $\mathfrak{o}[t_1,\dots,t_N]$ is still surjective, hence the claim.
\end{proof}

\subsection*{Choosing linear projections}
In the next step, we want to find the Zariski-open subset $W\subseteq \mathfrak{E}_n$ parametrizing the linear projections $p_{\underline{u}}$ in Theorem \ref{thm: gabber presentation}. 
To do so, let us first make one further reduction:

\begin{reduction}\label{red: closed point}
By Proposition~\ref{prop: kay density lemma}, there is a Nisnevich-neighbourhood $(X^\prime,z^\prime)\rightarrow(X,z)$ and a closed embedding ${X}^\prime \hookrightarrow \mathbb{A}_{\underline{x},S}$ such that $Z^\prime = Z\times_XX^\prime$ is fibrewise dense in its Zariski-closure $\bar{Z}^\prime$ in $\mathbb{P}_{\underline{x},S}$.
Let $z_0$ be a specialization of $z$ in the image of $X^\prime \rightarrow X$.
We can find a point $z_0^\prime$ in $Z^\prime$ such that $k(z_0^\prime) = k(z_0)$, \ie~$(X^\prime,z_0^\prime) \rightarrow (X,z_0)$ is a Nisnevich neighbourhood, too.
The Nisnevich-localization $(X^\prime,z^\prime)\rightarrow(X,z)$ will be the only non-Zariski-localization in the proof of Theorem~\ref{thm: gabber presentation}.
Thus we may assume that $z$ is a closed point in the following.
Further, from now on we may identify $X^\prime = X$.  
\end{reduction}

The Zariski-open subset $W\subseteq \mathfrak{E}_n$ in Theorem~\ref{thm: gabber presentation} will be provided in the following proposition.

\begin{proposition}\label{prop: gabber A}
Let ${X} = \Spec({A})/S$ be a connected smooth affine $S$-scheme of finite type, fibrewise of pure dimension $n$, $f$ an element in ${A}$ which is non-zero in ${A}\otimes_{\mathfrak{o}}\mathbb{F}$ and ${Z} = \Spec({B}={A}/f) \hookrightarrow {X}$ the closed embedding.
Let $z$ be a closed point in the special fibre of ${Z}$.
Suppose there is a closed embedding $i_0\colon {X}\hookrightarrow \mathbb{A}_{\underline{x},S}$ such that ${Z}$ is fibrewise dense over $S$ inside its closure $\bar{{Z}}$ in $\mathbb{P}_{\underline{x},S}$. 
\\
Then there is a Zariski-open subset $W\subseteq \mathfrak{E}_n$ with $W(\mathfrak{o}) \neq \emptyset$, \st\ for all $\underline{u} \in W(\mathfrak{o})$ the following holds:
\begin{enumerate}
 \item\label{prop: gabber A 1} $p_{(u_1,\dots,u_{n-1})}\vert_{{Z}}\colon {Z} \rightarrow \mathbb{A}_{t_1,\dots,t_{n-1},S}$ is finite,
 \item\label{prop: gabber A 2} $p_{\underline{u}}$ is \'{e}tale at all points of $F= p_{(u_1,\dots,u_{n-1})}^{-1}(p_{(u_1,\dots,u_{n-1})}(z))\cap {Z}$ and
 \item\label{prop: gabber A 3} $p_{\underline{u}}\vert_{F}\colon F\rightarrow p_{\underline{u}}(F)$ is radicial.
\end{enumerate}
\end{proposition}

Let us first fix the following notation.

\begin{remark}\label{rem: reduction map}
For ${Y}/S$ a smooth scheme, denote by ${\rm red}\colon {Y}(\mathfrak{o}) \rightarrow {Y}(\mathbb{F}) = {Y}_\sigma(\mathbb{F})$ the reduction map we get by pre-composing with the closed point $\sigma$.
Because $\mathfrak{o}$ is henselian and ${Y}/S$ smooth, this reduction map is always surjective.
\end{remark}

\begin{proof}[Proof of Proposition~\ref{prop: gabber A}]
We divide Proposition~\ref{prop: gabber A} into two parts:
Lemma~\ref{lem: finitness} will provide an open $W_1$ of $\mathfrak{E}_n$ such that $W_1(\mathfrak{o})$ is non-empty and every $\underline{u}$ in $W_1(\mathfrak{o})$ satisfies claim~\eqref{prop: gabber A 1} while Lemma~\ref{lem: etale and radicial} will provide an open $W_2$ such that $W_2(\mathfrak{o})$ is non-empty and every $\underline{u}$ in $W_2(\mathfrak{o})$ satisfies claims~\eqref{prop: gabber A 2} and~\eqref{prop: gabber A 3} in Proposition~\ref{prop: gabber A}.
The intersection $W=W_1\cap W_2$ has all the properties claimed by Proposition~\ref{prop: gabber A}.
For the non-emptiness of $W(\mathfrak{o})$, recall that the reduction map $W(\mathfrak{o}) \twoheadrightarrow W(\mathbb{F})$ is surjective and $W_1(\mathbb{F})  \cap W_2(\mathbb{F})$ is non-empty as the special fibre of $W_1\cap W_2$ is a non-empty open subscheme of an affine space over the infinite field $\mathbb{F}$.
\end{proof}

\begin{lemma}\label{lem: finitness}
Under the assumptions of Proposition~\ref{prop: gabber A}, there is a Zariski-open subset $W_1\subseteq \mathfrak{E}_n$ with $W_1(\mathfrak{o}) \neq \emptyset$, \st\ for all $\underline{u} \in W_1(\mathfrak{o})$ the restriction $p_{(u_1,\dots,u_{n-1})}\vert_{{Z}}\colon {Z} \rightarrow \mathbb{A}_{t_1,\dots,t_{n-1},S}$ is finite.
\end{lemma}

\begin{lemma}\label{lem: etale and radicial}
Under the assumptions of Proposition~\ref{prop: gabber A}, there is a Zariski-open subset $W_2\subseteq \mathfrak{E}_n$ with $W_2(\mathfrak{o}) \neq \emptyset$, \st\ $p_{\underline{u}}$ is \'{e}tale at all points of $F$ and $p_{\underline{u}}\vert_{F}\colon F\rightarrow p_{\underline{u}}(F)$ is radicial for all $\underline{u} \in W_2(\mathfrak{o})$.
\end{lemma}

\begin{proof}[Proof of Lemma~\ref{lem: finitness}]
This is just a version of the arguments leading to \cite[Prop.~1.1]{Grayson78}:
Recall that the $j^{\rm th}$-factor $\mathbb{A}_{\underline{x},S}^{\vee,(j)}$ of $\mathfrak{E}_n$ is $\mathbb{A}_{x_{1,j}^\ast,\dots,x_{N,j}^\ast,S}$, \ie, $\mathfrak{E}_n = \mathbb{A}_{\{x_{i,j}^\ast \vert 1\leq i \leq N, 1\leq j \leq n\},S}$.
Define
\begin{center}
 $\mathbb{L}:= V_+(x_0,\sum_jx_{i,j}^\ast\otimes x_j \vert 1\leq i <n) \subseteq \mathfrak{E}_n \times_S H_\infty$ and $\bar{{Z}}_\infty := \bar{{Z}}\cap H_\infty$.
\end{center}
Here, $H_\infty=V_+(x_0)\subset\mathbb{P}_{\underline{x},S}$ is the hyperplane at infinity.
By construction, $\mathbb{L}\rightarrow \mathfrak{E}_n$ has fibre $\mathbb{L}_{\underline{u}} = L_{(u_1,\dots,u_{n-1})}$ over $\underline{u}\in \mathfrak{E}_n(\mathfrak{o})$.
Since the projection ${\rm pr}\colon \mathfrak{E}_n \times_S H_\infty \rightarrow \mathfrak{E}_n$ is projective, hence closed,
\begin{equation*}
 W_1:= \mathfrak{E}_n \setminus {\rm pr}(\mathbb{L}\cap (\mathfrak{E}_n\times_S \bar{{Z}}_\infty))
\end{equation*}
is open.
Again by construction, for any $\underline{u}\in W_1(\mathfrak{o})$, $L_{(u_1,\dots,u_{n-1})} \cap \bar{{Z}} = \emptyset$, so $p_{(u_1,\dots,u_{n-1})}\vert_{{Z}}\colon {Z} \rightarrow \mathbb{A}_{t_1,\dots,t_{n-1},S}$ is finite by Lemma~\ref{lem: shafarevich}.
\\
It remains to show that $W_1(\mathfrak{o})\neq \emptyset$:
The reduction map ${\rm red}\colon W_1(\mathfrak{o}) \twoheadrightarrow W_1(\mathbb{F})$ is surjective, so we have to show $W_1(\mathbb{F}) = W_{1,\sigma}(\mathbb{F}) \neq \emptyset$.
The special fibre $W_{1,\sigma}$ equals $\mathfrak{E}_{n,\sigma} \setminus {\rm pr}(\mathbb{L}_\sigma\cap (\mathfrak{E}_{n,\sigma}\times_{\mathbb{F}} \bar{{Z}}_{\infty,\sigma}))$.
Further, ${Z}_\sigma \subset \bar{{Z}}_\sigma$ is dense by assumption so $\bar{{Z}}_\sigma$ is the closure $\overline{{Z}_\sigma}$ of ${Z}_\sigma$ inside $\mathbb{P}_{\underline{x},\mathbb{F}}$.
It follows that $\bar{{Z}}_{\infty,\sigma} = \overline{{Z}_\sigma} \cap H_{\infty,\sigma}$, \ie, we are in the situation of \cite[Prop.~1.1]{Grayson78} and $W_{1,\sigma}(\mathbb{F}) \neq \emptyset$.
\end{proof}

Lemma~\ref{lem: etale and radicial} can easily be derived from \cite[Lem.~3.4.1, Lem.~3.4.2]{CTHK97} applied over the special fibre:

\begin{proof}[Proof of Lemma~\ref{lem: etale and radicial}]
As a closed embedding of smooth $S$-schemes, $i_0\colon {X} \hookrightarrow \mathbb{A}_{\underline{x},S}$ is regular.
Let $I=(f_1,\dots,f_{N-n})\unlhd \mathfrak{o}[\underline{x}]$ be the ideal of $i_0$ for $f_1,\dots,f_{N-n}$ a regular sequence.
Write $A=\mathcal{O}(X)$ over $\mathfrak{o}[\underline{t}]$ (via $p_{\underline{u}}$) as $\mathfrak{o}[\underline{t}][\underline{x}]/(f_i,u_j-t_j\vert i,j)$.
Then $p_{\underline{u}}$ is \'{e}tale at a point $x \in X$, if it is standard smooth around $x$, \ie, if the Jacobi-determinant
\begin{equation*}
{\rm det}\left( \left\{\frac{\partial f_i}{\partial x_s}\right\}_{i,s} \left\vert \left\{\frac{\partial (u_j - t_j)}{\partial x_s}\right\}_{j,s} \right)\right. = {\rm det}\left( \left\{\frac{\partial f_i}{\partial x_s}\right\}_{i,s} \left\vert \left\{\frac{\partial (u_j)}{\partial x_s}\right\}_{j,s} \right)\right.
\end{equation*}
is invertible in $\mathcal{O}_{X,x}$.
We may write the latter determinant as $df_1\wedge \dots df_{N-n}\wedge du_1\wedge \dots \wedge du_n$ in $\Omega_{\mathfrak{o}[\underline{x}]/\mathfrak{o}}^{N}\otimes_{\mathfrak{o}[\underline{x}]}\mathcal{O}_{X,x}$.
Since $X\hookrightarrow \mathbb{A}_{\underline{x},S}$ is a smooth pair, the conormal sequence
\begin{equation*}
 0\rightarrow
 I/I^2 \otimes_A \mathcal{O}_{X,x} \rightarrow
 \Omega_{\mathfrak{o}[\underline{x}]/\mathfrak{o}}^{1}\otimes_{\mathfrak{o}[\underline{x}]}\mathcal{O}_{X,x} \rightarrow
 \Omega_{A/\mathfrak{o}}^{1}\otimes_{A}\mathcal{O}_{X,x} \rightarrow
 0
\end{equation*}
is split exact and $\Omega_{\mathfrak{o}[\underline{x}]/\mathfrak{o}}^{N}\otimes_{\mathfrak{o}[\underline{x}]}\mathcal{O}_{X,x} = \bigwedge^{N-n}(I/I^2 \otimes_A \mathcal{O}_{X,x})\otimes_{\mathcal{O}_{X,x}} (\Omega_{A/\mathfrak{o}}^{n}\otimes_{A}\mathcal{O}_{X,x})$.
Note that $I/I^2$ is free over $A$ with basis given by the regular sequence $f_1,\dots,f_{N-n}$.
In particular, $f_1\wedge \dots \wedge f_n$ is invertible in $\bigwedge^{N-n}(I/I^2 \otimes_A \mathcal{O}_{X,x})=\mathcal{O}_{X,x}$ and
\begin{equation*}
 df_1\wedge \dots df_{N-n}\wedge du_1\wedge \dots \wedge du_n = (f_1\wedge \dots \wedge f_n)\otimes (d\bar{u}_1\wedge \dots \wedge d\bar{u}_n)
\end{equation*}
is invertible if and only if $d\bar{u}_1\wedge \dots \wedge d\bar{u}_n$ is invertible in $\Omega_{A/\mathfrak{o}}^{n}\otimes_{A}\mathcal{O}_{X,x} = \mathcal{O}_{X,x}$ for $\bar{u}_j$ the image of $t_j$ under $\mathfrak{o}[\underline{t}]\rightarrow A$.
By Nakayama, this is equivalent to $d\bar{u}_1\wedge\dots\wedge d\bar{u}_n \neq 0$ in $\Omega_{{X}/S}^n \otimes_{\mathcal{O}_{{X}}} k(x)$.
Suppose, $x$ is contained in $F$.
Since $z$ lies in the special fibre, so does $x$ and $\Omega_{{X}/S}^n \otimes_{\mathcal{O}_{{X}}} k(x) = \Omega_{{X}_\sigma/\mathbb{F}}^n \otimes_{\mathcal{O}_{{X}_\sigma}} k(x)$.
Summing up, $p_{\underline{u}}$ is \'{e}tale at $x\in F$ if $d\bar{u}_1\wedge\dots\wedge d\bar{u}_n \neq 0$ in $\Omega_{{X}_\sigma/\mathbb{F}}^n \otimes_{\mathcal{O}_{{X}_\sigma}} k(x)$.
Thus, we are in fact in the situation of \cite[Lem.~3.4.1]{CTHK97}, \ie, we get a non-empty open subset $\bar{W}_2^\prime \subseteq \mathfrak{E}_{n,\sigma} \otimes_{\mathbb{F}}\mathbb{F}^{\rm alg}$ with $p_{\underline{u}}$ \'{e}tale around $F$ for all $\underline{u}$ with ${\rm red}(\underline{u}) \in \bar{W}_2^\prime(\mathbb{F}^{\rm alg})$.
Here, $\mathbb{F}^{\rm alg}/\mathbb{F}$ is an algebraic closure.
\\
For the universal injectivity, observe that $p_{\underline{u}}\vert_F = p_{{\rm red}(\underline{u})}\vert_F$.
Thus, we are in the situation of \cite[Lem.~3.4.2]{CTHK97}, \ie, we get a non-empty open subset $\bar{W}_2^{\prime\prime} \subseteq \mathfrak{E}_{n,\sigma} \otimes_{\mathbb{F}}\mathbb{F}^{\rm alg}$ with $p_{\underline{u}}\vert_{F\otimes_{\mathbb{F}}\mathbb{F}^{\rm alg}}$ (universally) injective for all $\underline{u}$ with ${\rm red}(\underline{u}) \in \bar{W}_2^{\prime\prime}(\mathbb{F}^{\rm alg})$.
\\
Finally, by a standard descent argument (\cf~the proof of \cite[Lem.~3.4.3]{CTHK97}) for the intersection $\bar{W}_2^{\prime} \cap \bar{W}_2^{\prime\prime}$, we get an open subset $\bar{W}_2 \subseteq \mathfrak{E}_{n,\sigma}$ with $\bar{W}_2(\mathbb{F})\neq \emptyset$, such that $p_{\underline{u}}$ is \'{e}tale at all points of $F$ and $p_{\underline{u}}\vert_{F}$ is universally injective for all $\underline{u}$ with ${\rm red}(\underline{u}) \in \bar{W}_2(\mathbb{F})$.  
Let $W_2 \subseteq \mathfrak{E}_n$ be any open subset with special fibre $W_{2,\sigma} = \bar{W}_2$.
Then $\underline{u} \in W_2(\mathfrak{o})$ if and only if ${\rm red}(\underline{u}) \in \bar{W}_2(\mathbb{F})$ and $W_2(\mathfrak{o}) \neq \emptyset$, since the reduction map is surjective.
\end{proof}

\subsection*{Choosing neighbourhoods}
Fix a linear projection $p_{\underline{u}}$ for an $\mathfrak{o}$-point $\underline{u}$ in the open subset $W\subseteq \mathfrak{E}_n$ provided by Proposition~\ref{prop: gabber A}.
In the following, we will construct the open neighbourhoods $V$ and $U$ in Theorem~\ref{thm: gabber presentation}.

As in~\cite{CTHK97}, we will first secure $V \subseteq \mathbb{A}_{t_1,\dots,t_{n-1},S}$ in Lemma~\ref{lem: securing V} and an open neighbourhood $z\in U_1\subseteq p_{(u_1,\dots,u_{n-1})}^{-1}(V)$ covering Theorem~\ref{thm: gabber presentation} parts \eqref{maintheoremtwo} and \eqref{maintheoremfour} in Lemma~\ref{lem: securing U}.
If we define $U$ as the intersection of $U_1$ with the \'{e}tale locus of $p_{\underline{u}}$, the pair $V$ and $U$ will finally satisfy claims \eqref{maintheoremtwo}, \eqref{maintheoremthree} and \eqref{maintheoremfour} of Theorem~\ref{thm: gabber presentation}.
The proofs can almost literally be transferred from~\cite{CTHK97}. 

\begin{lemma}\label{lem: securing V} \emph{(\cf\ \cite[Lem.~3.5.1]{CTHK97})}
Under the assumptions of Proposition~\ref{prop: gabber A} and any choice of linear projection $\underline{u}$ in $W(\mathfrak{o})$, there is a Zariski-open neighbourhood $V \subseteq \mathbb{A}_{t_1,\dots,t_{n-1},S}$ of $p_{(u_1,\dots,u_{n-1})}(z)$ such that $p_{\underline{u}}$ is \'{e}tale in a Zariski-open neighbourhood of ${Z} \cap p_{(u_1,\dots,u_{n-1})}^{-1}(V)$ and restricts to a closed embedding $ {Z} \cap p_{(u_1,\dots,u_{n-1})}^{-1}(V) \hookrightarrow \mathbb{A}_V^1$.
\end{lemma}

\begin{proof}
We will get $V$ as $V_1\cap V_2$, where $V_1 \subseteq \mathbb{A}_{t_1,\dots,t_{n-1},S}$ is an open neighbourhood of $p_{(u_1,\dots,u_{n-1})}(z)$ \st\ $p_{\underline{u}}$ is \'{e}tale at all points of ${Z} \cap p_{(u_1,\dots,u_{n-1})}^{-1}(V_1)$ and $V_2 \subseteq \mathbb{A}_{t_1,\dots,t_{n-1},S}$ is an open neighbourhood such that $p_{\underline{u}}$ restricts to a closed embedding ${Z} \cap p_{(u_1,\dots,u_{n-1})}^{-1}(V_2) \rightarrow \mathbb{A}_{V_2}^1$.
\\
Let $U^\prime\subseteq {X}$ be the \'{e}tale locus of $p_{\underline{u}}$.
Since $\underline{u}\in W(\mathfrak{o})$, $U^\prime$ is an open neighbourhood of $F= p_{(u_1,\dots,u_{n-1})}^{-1}(p_{(u_1,\dots,u_{n-1})}(z))\cap {Z}$ in ${X}$ (\cf~Proposition~\ref{prop: gabber A}).
As $p_{(u_1,\dots,u_{n-1})}\vert_{{Z}}$ is finite, $p_{(u_1,\dots,u_{n-1})}({Z}\setminus U^\prime)$ is closed and we set
\[
V_1:= \mathbb{A}_{t_1,\dots,t_{n-1},S} \setminus p_{(u_1,\dots,u_{n-1})}({Z}\setminus U^\prime). 
\]
By construction, $V_1$ is a Zariski-open neighbourhood of the image $p_{(u_1,\dots,u_{n-1})}(z)$ of $z$ and moreover ${Z} \cap p_{(u_1,\dots,u_{n-1})}^{-1}(V_1) \subseteq {Z} \cap U^\prime$ is contained in the \'{e}tale locus $U^\prime$ of $p_{\underline{u}}$.
\\
To get the neighbourhood $V_2$, consider $p_{\underline{u}}\vert_{{Z}}\colon {Z} \rightarrow \mathbb{A}_{\underline{t},S} = \mathbb{A}_{t_n,\mathbb{A}_{t_1,\dots,t_{n-1},S}}$ as a family of maps over $\mathbb{A}_{t_1,\dots,t_{n-1},S}$.
Since ${Z}/\mathbb{A}_{t_1,\dots,t_{n-1},S}$ is finite, the property ``$p_{\underline{u}}\vert_{{Z}}$ is a closed embedding'' is Zariski-open in the base $\mathbb{A}_{t_1,\dots,t_{n-1},S}$ by Nakayama.
Thus we have to show that the fibre of this family
\[
p_{\underline{u}}\vert_{F}\colon {Z} \cap p_{(u_1,\dots,u_{n-1})}^{-1}(p_{(u_1,\dots,u_{n-1})}(z)) = F \rightarrow \mathbb{A}_{t_n,p_{(u_1,\dots,u_{n-1})}(z)} 
\]
over $p_{(u_1,\dots,u_{n-1})}(z)$ is a closed embedding.
But $p_{\underline{u}}(F) \subset \mathbb{A}_{t_n,p_{(u_1,\dots,u_{n-1})}(z)}$ is closed as a finite set of closed points and $p_{\underline{u}}\vert_F \colon F \rightarrow p_{\underline{u}}(F)$ is a closed embedding as by Proposition \ref{prop: gabber A}, it is radicial and $p_{\underline{u}}$ is \'{e}tale hence unramified at each point of $F$.
\end{proof}

\begin{lemma}\label{lem: securing U} \emph{(\cf\ \cite[Lem.~3.6.1]{CTHK97})}
Under the assumptions of Proposition~\ref{prop: gabber A} and any choice of linear projection $\underline{u}$ in $W(\mathfrak{o})$, let ${Z}^\prime := {Z} \cap p_{(u_1,\dots,u_{n-1})}^{-1}(V)$ and $U_1:= p_{(u_1,\dots,u_{n-1})}^{-1}(V) \setminus (p_{\underline{u}}^{-1}(p_{\underline{u}}({Z}^\prime)) \setminus {Z}^\prime)$.
Then $U_1 \subseteq p_{(u_1,\dots,u_{n-1})}^{-1}(V)$ is a Zariski-open neighbourhood of the point $z$, ${Z} \cap U_1 = {Z}^\prime$ and we have $p_{\underline{u}}^{-1}(p_{\underline{u}}({Z}^\prime)) \cap U_1 = {Z}^\prime$.
\end{lemma}

\begin{proof}
By definition of $U_1$, $z$ lies inside $U_1$, ${Z} \cap U_1 = {Z}^\prime$ and $p_{\underline{u}}^{-1}(p_{\underline{u}}({Z}^\prime)) \cap U_1 = {Z}^\prime$.
It remains to show that $U_1 \subseteq p_{(u_1,\dots,u_{n-1})}^{-1}(V)$ is open.
By Lemma~\ref{lem: securing V}, $p_{\underline{u}}$ restricts to a closed embedding ${Z}^\prime \rightarrow \mathbb{A}_{t_n,V}$, so $p_{\underline{u}}^{-1}(p_{\underline{u}}({Z}^\prime)) \subset p_{(u_1,\dots,u_{n-1})}^{-1}(V)$ is closed.
We need to show that ${Z}^\prime \subseteq p_{\underline{u}}^{-1}(p_{\underline{u}}({Z}^\prime))$ is open.
\\
To this end, consider the \'{e}tale locus $U^{\prime\prime}$ of $p_{\underline{u}}\vert_{p_{\underline{u}}^{-1}(p_{\underline{u}}({Z}^\prime))}\colon p_{\underline{u}}^{-1}(p_{\underline{u}}({Z}^\prime)) \rightarrow p_{\underline{u}}({Z}^\prime)$.
By Lemma~\ref{lem: securing V}, $p_{\underline{u}}$ is \'{e}tale at all points of ${Z}^\prime$.
Thus, the base change $p_{\underline{u}}\vert_{p_{\underline{u}}^{-1}(p_{\underline{u}}({Z}^\prime))}$ is still \'{e}tale at all points of ${Z}^\prime$, \ie, ${Z}^\prime$ is contained inside the open subset $U^{\prime\prime}\subseteq p_{\underline{u}}^{-1}(p_{\underline{u}}({Z}^\prime))$.
But ${Z}^\prime \rightarrow p_{\underline{u}}({Z}^\prime)$ is an isomorphism by Lemma~\ref{lem: securing V}, so both ${Z}^\prime/ p_{\underline{u}}({Z}^\prime)$ and $U^{\prime\prime} / p_{\underline{u}}({Z}^\prime)$ are \'{e}tale and hence ${Z}^\prime \subseteq U^{\prime\prime}$ is open.
\end{proof}

\section{\texorpdfstring{Objectwise stable $\AA^1$-connectivity}{Objectwise stable A1-connectivity}}
\label{chapterobjectwise}

\noindent
In this section, we derive connectivity results for homotopy presheaves (\ie\ ``objectwise'' connectivity results).
These are used in the proof of our main theorem in the following section. 
Moreover, we show the left completeness of the $\AA^1$-Nisnevich-local t-structure on $S^1$- and on $\PP^1$-spectra.
Throughout this section, let $S$ be an arbitrary noetherian scheme of finite dimension.

\subsection*{Results for $S^1$-spectra}
We start with objectwise connectivity results for $S^1$-spectra.

\begin{proposition}\label{absoluteconnectivity}
Let $U\in\Sm_S$ be a scheme of dimension $e$. Then for $E\in \SH^s_{S^1>i+e}(S)$, one has
\[
[\Sigma^\infty_{S^1}(U_\pt)[i], L^\Aeins E] = 0,
\]
where $L^\Aeins$ is a fibrant replacement functor for the stable $\AA^1$-Nisnevich-local model structure.
\end{proposition}

\begin{remark}
Proposition~\ref{absoluteconnectivity} gives a connectivity result for a $U$-section of the homotopy presheaf $[\Sigma^\infty_{S^1}(-_\pt)[i], L^\Aeins E]$ with respect to the dimension of $U$.
However, we are interested in a connectivity result depending only on the dimension of the base scheme $S$.
The price we have to pay for this is to sheafify the homotopy presheaf, \ie, eventually we are interested in connectivity results for the Nisnevich-stalks of the presheaf $[\Sigma^\infty_{S^1}(-_\pt)[i], L^\Aeins E]$.
Unfortunately we cannot apply Proposition~\ref{absoluteconnectivity} directly to the stalks as their dimension is unbounded.
\end{remark}

\begin{proof}[Proof of Proposition~\ref{absoluteconnectivity}]
We work with the explicit model $L^\infty\hspace{-2pt}$ of Lemma~\ref{specialreplacement} as an $\AA^1$-Nisnevich-local fibrant replacement functor $L^\Aeins$.
By homotopy-exactness of $L^\infty$, we have to show that
\[
[\Sigma^\infty_{S^1}(U_\pt), \underset{k\to \infty}{\hocolim}~ L^k  (E) ] = 0
\]
for $U\in\Sm_S$ of dimension $e$ and $E\in \SH^s_{S^1>e}$.
Since $\Sigma^\infty_{S^1}(U_\pt)$ is compact, every homotopy class in question is represented by some $\Sigma^\infty_{S^1}(U_\pt)\to L^k (E)$.
Hence, it suffices to show, that for every $k\geq 0$
\[
[\Sigma^\infty_{S^1}(U_\pt), L^k E] = 0.
\]
We argue by induction on $k\geq 0$ for all $U\in\Sm_S$ of dimension $e$ and all spectra $E\in \SH^s_{S^1>e}$ at once.\\
For $k=0$ the statement follows directly from Lemma~\ref{absoluteconnectivitylemma} below.
Let $k\geq 1$.
The distinguished triangle in Remark~\ref{rem: Lk exact triangle} induces the long exact sequence
\[
\cdots\to [\Sigma^\infty_{S^1}(U_\pt), L^{(k-1)} E)] \to [\Sigma^\infty_{S^1}(U_\pt),L^{k} E] \to [\Sigma^\infty_{S^1}(U_\pt \wedge \AA^1),L^{(k-1)} E[1]] \to\cdots .
\]
The abelian group on the left-hand side vanishes by the induction hypothesis on $k$. In order to see the vanishing of the right-hand side, we observe that
\[
U\sqcup \AA^1 \cong U_\pt \vee \AA^1 \to (U\times \AA^1)\sqcup \AA^1 \cong U_\pt\times \AA^1 \to U_\pt \wedge \AA^1
\]
and therefore $U_\pt \to (U\times \AA^1)_\pt \to U_\pt \wedge \AA^1$ is a homotopy cofibre sequence in $\sPre_\pt(S)$.
This yields a distinguished triangle after applying the left Quillen functor $\Sigma_{S^1}^\infty$.
Consider the long exact sequence obtained by an application of $[-,L^{(k-1)} E[1]]$ to this triangle.
It suffices to show the vanishing of both abelian groups $[\Sigma_{S^1}^\infty (U_\pt)[1], L^{(k-1)} E[1]]$ and $[\Sigma_{S^1}^\infty (U\times \AA^1)_\pt,L^{(k-1)} E[1]]$.
For the first group, this follows from the inductive hypothesis on $k$ and likewise for the second, since the dimension of $U\times \AA^1$ is $e+1$ and $E[1]\in \SH^s_{S^1>e+1}$.
\end{proof}

\begin{lemma}\label{absoluteconnectivitylemma}
Let $U\in\Sm_S$ be a scheme of dimension $e$.
Then for $D\in \SH^s_{S^1>i+e}(S)$, one has
\[
[\Sigma^\infty_{S^1}(U_\pt)[i], L^\simpl D] = 0.
\]
\end{lemma}
\begin{proof}
It suffices to show that $[\Sigma^\infty_{S^1}(U_\pt), L^\simpl D]=0$ for $D\in \SH^s_{S^1>e}$.
Indeed, $[\Sigma^\infty_{S^1}(U_\pt)[i], L^\simpl D]\cong [\Sigma^\infty_{S^1}(U_\pt), L^\simpl\hspace{1pt}(D[-i])]$ as $L^s$ is homotopy-exact.
Recall that the Nisnevich-cohomological dimension is bounded by the Krull-dimension, \ie, for any sheaf $G$ of abelian groups on $\Sm_S$ and $n>\dim(U)$, we have $[\Sigma^\infty_{S^1}(U_\pt), L^\simpl HG[n]]=H^n_{Nis}(U,G)=0$ (see \eg\ \cite[Lem.~E.6.c]{TT90}).
\\
By the left completeness of the Nisnevich-local structure, there is a filtration
\[
\xymatrix@R=10pt@C=7pt{
0\simeq \underset{n\to\infty}{\holim} L^\simpl D_{ \geq n} \ar@{->}[r] & \ldots \ar@{->}[r]  & L^\simpl D_{\geq e+2} \ar@{->}[r] &L^\simpl D_{\geq e+1} = L^\simpl D\\
                                       & \ldots            & L^\simpl H\pi_{e+2}(D)[e+2] \ar@{<-}[u] & L^\simpl H\pi_{e+1}(D)[e+1] \ar@{<-}[u]
}
\]
and a surjection $0=[\Sigma^\infty_{S^1}(U_\pt),\holim_n L^\simpl D_{ \geq n}]\twoheadrightarrow \lim_n [\Sigma^\infty_{S^1}(U_\pt), L^\simpl D_{ \geq n}]$ by the Milnor-$\lim^1$-sequence.
Hence, $\lim_n [\Sigma^\infty_{S^1}(U_\pt), L^\simpl D_{\geq n}]=0$. For $i\geq 1$, there is a long exact sequence
\[
\cdots\to [\Sigma^\infty_{S^1}(U_\pt), L^\simpl D_{\geq e+i+1}] \to  [\Sigma^\infty_{S^1}(U_\pt), L^\simpl D_{\geq e+i}]\to [\Sigma^\infty_{S^1}(U_\pt), L^\simpl H\pi_{e+i}(D)[e+i]]\to\cdots
\]
where the abelian group on the right-hand side is zero by the result on the Nisnevich-cohomological dimension mentioned above.
For this reason, the projection $\lim [\Sigma^\infty_{S^1}(U_\pt), L^\simpl D_{\geq n}] \twoheadrightarrow [\Sigma^\infty_{S^1}(U_\pt), L^\simpl D_{\geq e+1}]$ is surjective and therefore we get $[\Sigma^\infty_{S^1}(U_\pt), L^\simpl D_{\geq e+1}]=0$ as desired. 
\end{proof}

\begin{corollary}\label{absoluteconnectivitylemmanobp}
Let $U\in\Sm_S$ be an $S$-pointed scheme of dimension $e$.
For $D\in \SH^s_{S^1>i+e}(S)$, one has
\[
[\Sigma^\infty_{S^1}(U)[i], L^\simpl D] = 0.
\]
\end{corollary}
\begin{proof}
The basepoint $s\colon S\to U$ is a splitting of the structure morphism $p\colon U\to S$.
In particular, $\dim(U)\geq\dim(S)$.
Consider the distinguished triangle
\[
\Sigma^\infty_{S^1}(S_\pt) \to \Sigma^\infty_{S^1}(U_\pt) \to \Sigma^\infty_{S^1}(U) \to \Sigma^\infty_{S^1}(S_\pt)[1].
\]
If $\dim(U)>\dim(S)$ then the assertion follows from the previous Lemma \ref{absoluteconnectivitylemma} applied to the entries $\Sigma^\infty_{S^1}(U_\pt)$ and $\Sigma^\infty_{S^1}(S_\pt)[1]$ of the triangle.
Now we consider the case $\dim(U)=\dim(S)$.
Because of the splitting $s\colon S\to U$, $p$ is surjective.
Since $p$ is smooth of relative dimension zero, it follows that $p$ is \'etale.
Thus, the section $s$ itself is \'etale. As it is also a closed immersion, the image of $s$ is a component of $U$, \ie, $U\cong (U')_\pt$ for some $U'\in\Sm_S$ with $\dim(U')\leq \dim(U)$.
The result then follows from the previous Lemma \ref{absoluteconnectivitylemma} applied to $\Sigma^\infty_{S^1}(U'_\pt)$.
\end{proof}

We get the following analogue as a corollary to Proposition~\ref{absoluteconnectivity}.

\begin{corollary}
Let $U\in\Sm_S$ be an $S$-pointed scheme of dimension $e$.
For $E\in \SH^s_{S^1>i+e}(S)$, one has
\[
[\Sigma^\infty_{S^1}(U)[i], L^\Aeins E] = 0.
\] 
\end{corollary}
\begin{proof}
The proof is literally the same as the proof of Corollary~\ref{absoluteconnectivitylemmanobp} using Proposition~\ref{absoluteconnectivity} instead of Lemma~\ref{absoluteconnectivitylemma}. 
\end{proof}

\begin{corollary}\label{tstructurenondegenerate}
The $\AA^1$-Nisnevich-local t-structure on $\SH^\Aeins_{S^1}(S)$ is left complete and hence non-de\-generate.
In particular,
\[
L^\Aeins \underset{n\to\infty}{\holim} (E_{\leq n}) \sim \underset{n\to\infty}{\holim} L^\Aeins (E_{\leq n}).
\]
\end{corollary}
\begin{proof}
First note, that the truncation functors of the $\AA^1$-Nisnevich-local t-structure are (after inclusion to $\SH_{S^1}^\simpl$) given by $L^\Aeins((-)_{\leq n})$ (\cf~Proposition~\ref{tstructureconstruction}).
Consider a spectrum $E\in\Spt_{S^1}(S)$.
To see that $L^\Aeins E \to\holim_{n} L^\Aeins(E_{\leq n})$ is an isomorphism in $\SH^\simpl_{S^1}$, we may equivalently show $\holim_{n} L^\Aeins(E_{\geq n})\simeq 0$ which is implied by the triviality of the group $\pi_i (\holim_{n} L^\Aeins(E_{\geq n}))(U)$ for every integer $i$ and every $U\in\Sm_S$.
Equivalently, we show that $\pi_i \holim_{n} (L^\Aeins(E_{\geq n})(U))$ is trivial.
Proposition~\ref{absoluteconnectivity} yields $[\Sigma^\infty_{S^1} U_+[i],L^\Aeins(E_{\geq n}) ] = 0$ for all integers $n> i+\dim(U)$.
Hence, we obtain $\lim_n \pi_i(L^\Aeins(E_{\geq n})(U)) = 0$.
Using Milnor's $\lim^1$-sequence, it follows that the group $\pi_i \holim_{n} (L^\Aeins(E_{\geq n})(U))$ is trivial:
Indeed, the $\lim^1$-term is trivial as the occurring groups are eventually zero.
\end{proof}

\subsection*{Results for $\PP^1$-spectra}
In this subsection, we show some analogous staments to those of the preceeding section for $\PP^1$-spectra. 
The results of this subsection are not needed for the rest of the paper but are of independent interest.

\begin{proposition}\label{absoluteconnectivityp1}
Let $U\in\Sm_S$ be a scheme of dimension $e$. Then for $E\in \SH_{> i+e}(S)$, one has
\[
[\Sigma^\infty_{\PP^1}(U_\pt)[i]\langle q\rangle , E]_{\SH} = 0
\]
for all $q\in\ZZ$.
\end{proposition}
\begin{proof}
Set $\mathcal{F}:=[\Sigma^\infty_{\PP^1}(U_\pt)[i]\langle q\rangle , -]_{\SH}$ for abbreviation.
By the construction in Proposition \ref{tstructureconstruction}, the class $\SH_{> i+e}$ is generated under extensions, (small) sums and cones from $\S[i+e+1]$.
If $E$ is obtained from an extension $E'\to E\to E''$ and $\mathcal{F}$ vanishes on $E'$ and $E''$, then it also vanishes on $E$.
If $E$ is a (small) sum of objects $E'_\alpha$ on which $\mathcal{F}$ vanishes, we use the homotopy-compactness of $\Sigma^\infty_{\PP^1} (U_\pt)[i]\langle q\rangle$ to conclude that $\mathcal{F}(E)=0$.
Suppose that $E$ sits in a distinguished triangle $E'\to E''\to E\to E'[1]$ and we know the vanishing of $\mathcal{F}$ on $E''$ and $E'[1]$, then we know it on $E$.
Summing up, it suffices to show that $\mathcal{F}(\S[n])= 0$ for all $n\geq i+e+1$, \ie, $[\Sigma^\infty_{\PP^1} (U_\pt) [i]\langle q\rangle ,~ \Sigma_{\PP^1}^\infty (V_\pt)[n] \langle q' \rangle ]_{\SH} = 0$ for all $V\in\Sm_S$ and $q'\in\ZZ$. We compute
\[
 \begin{array}{rcl}
 [\Sigma^\infty_{\PP^1} (U_\pt) [i]\langle q\rangle ,~ \Sigma_{\PP^1}^\infty (V_\pt)[n] \langle q' \rangle ]_{\SH} 
  &\cong & [\Sigma^\infty_{\PP^1} (U_\pt) [i-n],~ \Sigma_{\PP^1}^\infty (V_\pt) \langle q'-q \rangle ]_{\SH} \vspace{3pt} \\
  &\cong & [\Sigma^\infty_{S^1} (U_\pt) [i-n],~ \Omega_{\GG_m}^\infty (\Sigma_{\GG_m}^\infty\Sigma_{S^1}^\infty (V_\pt) \langle q'-q \rangle )] \vspace{3pt} \\
  &\cong & [\Sigma^\infty_{S^1} (U_\pt) [i-n],~ \colim_k \Omega_{\GG_m}^k L^\Aeins (\Sigma_{S^1}^\infty (V_\pt) \wedge \GG_m^{\wedge (k+q'-q)})] \vspace{3pt} \\
  &\cong & \colim_k ~ [\Sigma^\infty_{S^1} (U_+) [i-n] \wedge \GG_m^{\wedge k},~ L^\Aeins (\Sigma_{S^1}^\infty (V_\pt) \wedge \GG_m^{\wedge (k+q'-q)})] \vspace{3pt} \\
  \end{array}
\]
where the last isomorphism is due to compactness of $\Sigma^\infty_{S^1} (U_\pt)$.
Now we use that $\AA^1$-Nisnevich-locally there is an equivalence $\GG_m\sim \PP^1[-1]$.
Hence, it suffices to show that for all but finitely many $k\geq 0$ (and in particular, we may assume $k+q'-q\geq 0$), one has
\[
\begin{array}{lcl}
     [\Sigma^\infty_{S^1} (U_\pt) [i-n] \wedge \GG_m^{\wedge k},~ L^\Aeins (\Sigma_{S^1}^\infty (V_{\pt}) \wedge \GG_m^{\wedge (k+q'-q)})] &\cong &\\
     {[}\Sigma^\infty_{S^1} (U_\pt\wedge ({\PP^1})^{\wedge k}) [i-n-k],~ L^\Aeins \Sigma_{S^1}^\infty (V_{\pt} \wedge \GG_m^{\wedge (k+q'-q)})]  &= & 0.
\end{array}
\]
By the same arguments as in the proof of Proposition \ref{absoluteconnectivity}, this is implied by the vanishing of the group $[\Sigma^\infty_{S^1} (U\times \PP^k)_+ [i-n-k],~ L^\Aeins \Sigma_{S^1}^\infty (V_\pt \wedge \GG_m^{\wedge (k+q'-q)})]$.
Since the spectrum $ \Sigma_{S^1}^\infty (V_\pt \wedge \GG_m^{\wedge (k+q'-q)})$ is in $\SH^s_{S^1\geq 0}$, the result follows from Proposition \ref{absoluteconnectivity} as the scheme $U\times \PP^k$ has dimension $e+k$.
\end{proof}

\begin{corollary}\label{homotopytstructurenondegenerate}
Let $S$ be a noetherian scheme of finite Krull-dimension.
Then the homotopy t-structure on the motivic homotopy category $\SH(S)$ is left complete and hence non-degenerate.
\end{corollary}
\begin{proof}
Let $E\in \SH$.
We have to show that the canonical morphism $E\to\holim E_{\leq n}$ is an isomorphism in $\SH$.
Equivalently, we may show that $\holim E_{\geq n}\simeq 0$.
By \cite[Thm.~7.3.1]{Hovey99}, this is implied by the vanishing of the homotopy classes
\[
[\Sigma^\infty_{\PP^1} (U_+) [i]\langle q\rangle , \underset{n\to\infty}{\holim} ~ E_{\geq n} ]
\]
in $\SH$ for all $U\in\Sm_S$ and all $i,q\in\ZZ$. Using Milnor's $\lim^1$-sequence as in Corollary~\ref{tstructurenondegenerate}, this, in turn, is implied by the following statement:
For all $U\in\Sm_S$ and $i,q\in\ZZ$ there exists an integer $n_0$ with $[\Sigma^\infty_{\PP^1} (U_+) [i]\langle q\rangle , E_{\geq n} ]  = 0$ for all $n\geq n_0$.
Setting $n_0:=i+\dim(U)$, this precisely is the preceeding Proposition \ref{absoluteconnectivityp1}.
\end{proof}

\section{\texorpdfstring{Stalkwise stable $\AA^1$-connectivity}{Stalkwise stable A1-connectivity}}

\noindent
In this section, we derive our main connectivity result for homotopy sheaves (\ie~a ``stalkwise'' connectivity result).
We formulate the \emph{shifted stable $\AA^1$-connectivity property} on the base scheme and show that this property holds for every Dedekind scheme with infinite residue fields.

\subsection*{Stable $\AA^1$-connectivity}

Let us recall the following property on a base scheme $S$ introduced by Morel in \cite[Def.~1]{Morel05}.

\begin{defi}
\label{defi:stable}
A noetherian scheme $S$ of dimension $d$ has the \emph{stable $\AA^1$-connectivity property}, if for every integer $i$ and every spectrum $E$ in $\SH^s_{S^1\geq i}(S)$, the $\AA^1$-localization $L^\Aeins E$ is contained in $\SH^s_{S^1\geq i}(S)$.
\end{defi}

\begin{theorem}[Morel, {\cite[Thm.~6.1.8]{Morel05}}]\label{morelconnectivity}
If $S$ is the spectrum of a field, then $S$ has the stable $\AA^1$-connectivity property.
\end{theorem}

\begin{corollary}
If $S$ is the spectrum of a field, then $\SH^\Aeins_{S^1\geq 0}(S)= \SH^{\Aeins,\pi}_{S^1\geq 0}(S)$.
\end{corollary}

\begin{remark}\label{counterexample}
In \cite{Ayoub06}, Ayoub gave examples of base schemes that do not have the stable $\AA^1$-con\-nectivity property:
Let $S/k$ be a connected normal surface over $k$ an algebraically closed field, regular away from one closed singular point $s$.
Let $S^\prime \rightarrow S$ be a resolution with exceptional divisor $E$ and let $E_{\rm red}$ be the underlying reduced subscheme.
Then by \opcit~Corollary~3.3, $S$ does not have the stable $\AA^1$-connectivity property if ${\rm Pic}_{E_{\rm red}}$ is not $\AA^1$-invariant.
Here, ${\rm Pic}_{E_{\rm red}}$ is the Nisnevich-sheafification of the presheaf $U\mapsto {\rm Pic}(U\times_sE_{\rm red})$ on $\Sm_{k(s)}$. 
A family of concrete examples for such a surface $S$ (due to Barbieri-Viale) is given in the example in Section 3 of \opcit~as hypersurfaces of $\PP_k^3$.
Even worse, it follows from \opcit~Lemma~1.3 that no $\PP_k^n$ for $n\geq 3$ has the stable $\AA^1$-connectivity property.
\end{remark}

\subsection*{Towards stable $\AA^1$-connectivity}

We saw in Remark~\ref{counterexample}, that connectivity may drop for general base schemes.
Thus, it is an interesting question if, for a given base scheme $S$, there is at least \emph{some} uniform bound $r$ for the loss of connectivity, \ie, for $E$ an $i$-connected spectrum, the $\AA^1$-localization $L^\Aeins E$ is at least $(i-r)$-connected.
In this subsection, we want to discuss a general recipe for finding such a bound, based on Morel's original work over a field.

\begin{proposition}\label{mainproposition}
Let $S$ be a noetherian scheme of finite Krull-dimension and let $r\geq 0$ be an integer.
Let $E\in\SH^s_{S^1\geq i}(S)$ be a spectrum. 
Suppose for all $V\in\Sm_S$ and all $f\in[\Sigma^\infty_{S^1}V_\pt ,L^\Aeins E]$, Nisnevich-locally in $V$, there exists a Zariski-open $W\hookrightarrow V$ such that
\begin{enumerate}
 \item\label{mainprop1} $f|_{\Sigma^\infty_{S^1} W_\pt}=0$, and
 \item\label{mainprop2} $\pi_0^\Aeins(V/W)=0$.
\end{enumerate}
Then $L^\Aeins E\in \SH^s_{S^1\geq i-r}(S)$.
\end{proposition}
\begin{proof}
After shifting, we can assume that $i=r$.
We have to show that the sheaf $\pi_0^\Aeins(E)$ is trivial.
Take a connected scheme $V\in\Sm_S$ with structure morphism $p\colon V\to S$ and a point $v\in V$.
It suffices to show that the Nisnevich-stalk of $\pi_0^\Aeins(E)$ at $(V,v)$ is trivial.
Let $f_{(V,v)}$ be a germ in this stalk.
Possibly refining $(V,v)$ Nisnevich-locally, we may assume that $f_{(V,v)}$ is induced by an element $f\in[\Sigma^\infty_{S^1}V_\pt ,L^\Aeins E]$.
After a further Nisnevich-refinement of $(V,v)$, we find a Zariski-open $W\hookrightarrow V$ satisfying properties~\eqref{mainprop1} and \eqref{mainprop2}.
The homotopy cofibre sequence $W_\pt\to V_\pt \to V/W$ induces a long exact sequence
\[
\cdots\to [\Sigma^\infty_{S^1}V/W ,L^\Aeins E] \to [\Sigma^\infty_{S^1}V_\pt ,L^\Aeins E] \to [\Sigma^\infty_{S^1}W_\pt ,L^\Aeins E]\to\cdots.
\]
Since the restriction of $f$ to $\Sigma^\infty_{S^1}W_\pt$ is trivial by~\eqref{mainprop1}, $f$ is the image of an element in the group $[\Sigma^\infty_{S^1}V/W ,L^\Aeins E]$, \ie, a morphism $g\colon V/W\to (L^\Aeins E)_0$ in the (unstable) objectwise (pointed) homotopy category.
We want to show the triviality of the germ $f_{(V,v)}$, so it is enough to show that $\pi_0(g)$ is trivial.
As the adjunction $(\Sigma^\infty_{S^1},(-)_0)$ is a Quillen-adjunction for the $\AA^1$-Nisnevich-local model, $(L^\Aeins E)_0$ is $\AA^1$-Nisnevich-local.
Therefore, the morphism $g$ factors through $h\colon L^\Aeins(V/W)\to (L^\Aeins E)_0$ and it suffices to show that $\pi_0(h)$ is trivial which follows from the assumption~\eqref{mainprop2}.
\end{proof}

In the following, let us discuss how to obtain assumptions~\eqref{mainprop1} and~\eqref{mainprop2} from the previous proposition.
We start with assumption~\eqref{mainprop2}.

\medskip

Let us first recall the following construction.
The singular functor $\Sing\colon \sPre(S)\to\sPre(S)$ is given on $U$-sec\-tions by the diagonal of the bisimplicial set $\Sing(F)(U)=F_{\mbox{\larger[-20]$\bullet$}}(\Delta\hspace{-7pt}\mbox{\larger[-2]$\Delta$}^{\mbox{\larger[-20]$\bullet$}}\times U)$ where $\Delta\hspace{-7pt}\mbox{\larger[-2]$\Delta$}^{\mbox{\larger[-20]$\bullet$}}$ denotes the standard cosimplicial object in $\Sm_S$ (\cf~\cite[Ch.~2.3.2]{MV99} in the analogous situation for simplicial sheaves).
An infinite alternating composition of a Nisnevich-local fibrant replacement functor $L^s$ and $\Sing$ yields an $\AA^1$-Nisnevich-local fibrant replacement functor $L^\Aeins$ in the unstable setting. We refer to \cite[Ch.~2.3.2]{MV99} for details of this well-known construction.

\smallskip

Note that for Lemma~\ref{connectivityassuminggabber} and its Corollary~\ref{a1connectdensecorollary}, we work in the unstable setting.

\begin{lemma}\label{connectivityassuminggabber}
Let $V\in\Sm_S$ be an irreducible scheme and $W\hookrightarrow V$ a non-empty open subscheme.
Let $Z=(V\setminus W)_{\rm red}$ be the reduced complement.
Suppose moreover, that each point $v$ of $V$ admits a Nisnevich-neighbourhood $V^\prime$ (with pullback $W^\prime$ and $Z^\prime$ to $V^\prime$) together with an \'{e}tale map $p\colon V^\prime \rightarrow \AA^1_Y$ in $\Sm_S$ with $Z^\prime \to Y$ finite such that
\[
  \xymatrix{
   W^\prime \ar[r] \ar[d] &
   V^\prime \ar[d]^-p
  \\
   \AA_Y^1 \setminus p(Z^\prime) \ar[r] &
   \AA_Y^1
  }
\]
is a Nisnevich-distinguished square.
Then $\pi_0(\Sing(a_{\rm Nis}(V/W)))$ is trivial for $a_{\rm Nis}$ the Nisnevich-sheafi\-fication.
\end{lemma}
\begin{proof}
We follow the proof of \cite[Lem.~6.1.4]{Morel05}.
Since a simplicial presheaf $F$ has the same zero-simplices as the simplicial presheaf $\Sing(F)$, there is an epimorphism $[(-)_\pt,F]\twoheadrightarrow [(-)_\pt,\Sing(F)]$ of presheaves.
As Nisnevich-sheafification preserves epimorphisms, we get a natural epimorphism $\pi_0(F)\twoheadrightarrow \pi_0(\Sing(F))$ of sheaves.
Applying this to the discrete simplicial presheaf $F=a_{\rm Nis}(V/W)$ and pre-composing with the epimorphism $V=a_{\rm Nis}V\twoheadrightarrow a_{\rm Nis}(V/W)$ of sheaves, we get a natural epimorphism 
\[
 V\twoheadrightarrow a_{\rm Nis}(V/W) = \pi_0(a_{\rm Nis}(V/W)) \twoheadrightarrow \pi_0(\Sing(a_{\rm Nis}(V/W)))
\]
of sheaves. Hence, it suffices to show that for each point $v\in V$ there exists a Nisnevich-neighbourhood $V'$ such that $V'\to \pi_0(\Sing(a_{\rm Nis}(V'/W')))$ is zero where $W':=W\times_V V'$.
\\
Let $v\in V$ be a point and choose the Nisnevich-neighbourhood $V'$ from the assumption of the proposition.
In particular, $V'/W' \rightarrow \AA^1_Y / (\AA^1_Y\setminus Z{'})$ is an isomorphism of Nisnevich sheaves, so we may assume $V^\prime = \AA^1_Y$ with closed $Z^\prime \hookrightarrow \AA^1_Y$ and $Z^\prime\to Y$ finite.
By finiteness, the morphism $Z'\to \AA^1_Y\hookrightarrow\PP^1_Y$ is proper and hence a closed immersion. Therefore we get a diagram
\[
\xymatrix@M=4pt{
{\AA}^1_Y\setminus Z{'}\ar[r]\ar[d] & {\AA}^1_Y \ar@{->}[r]^(0.27){q} \ar[d]^j & a_{\rm Nis}\left(\bigslant{{\AA}^1_Y}{{\AA}^1_Y\setminus Z{'}}\right)\ar@{->}[d]^\cong\\
{\PP}^1_Y\setminus Z{'}\ar[r]& {\PP}^1_Y \ar@{->}[r]^(0.28){q'} & a_{\rm Nis}\left(\bigslant{{\PP}^1_Y}{{\PP}^1_Y\setminus Z{'}}\right)\\
}
\]
where the right vertical morphism is an isomorphism of Nisnevich sheaves as the left-hand square is a Zariski- and therefore a Nisnevich-distinguished square.
\\
There exists an elementary $\AA^1$-homotopy $Y\times\AA^1\to \PP^1_Y$ from the zero-section $s_0\colon Y\to\AA^1_Y\hookrightarrow\PP^1_Y$ to the section $s_\infty\colon Y\to\PP^1_Y$ at infinity and the latter factorizes over ${\PP}^1_Y\setminus Z{'}$.
As by \cite[Lem.~2.3.6]{MV99} the functor $\Sing$ turns elementary $\AA^1$-homotopies into objectwise homotopies, the maps $\Sing(s_0)$ and $\Sing(s_\infty)$ are identified in the objectwise homotopy category.
The morphism $\Sing(Y)\xrightarrow{\sim} \Sing(\AA^1_Y)$ is an objectwise weak equivalence by \cite[Cor.~2.3.5]{MV99}. Hence, the composition $\Sing(q'\circ j)$ is the constant map to the point in the objectwise homotopy category.
It follows that the same is true for $\Sing(q)$, as desired. Note that the cited arguments of \cite{MV99} are valid for the objectwise structure on simplicial presheaves. 
\end{proof}

Using the epimorphism $\pi_0(\Sing(a_{\rm Nis}(V/W)))\twoheadrightarrow \pi_0^\Aeins(V/W)$ of sheaves \cite[Cor.~2.3.22]{MV99}, we get:

\begin{corollary}\label{a1connectdensecorollary}
In the situation of the previous Lemma~\ref{connectivityassuminggabber}, we have
\[
\pi_0^\Aeins(V/W)= 0.
\]
\end{corollary}

\begin{remark}\label{counterexampleV/W}
As explained in \cite[Rem.~6.1.5]{Morel05}, $\pi_0^\Aeins(V/W)= 0$ might fail for arbitrary open subschemes $W\hookrightarrow V$.
For example, let $S$ be the spectrum of a local ring with closed point $i\colon\sigma\hookrightarrow S$ and open complement $j\colon W\hookrightarrow S$. Set $V:=S$ and consider the $\AA^1$-Nisnevich-local homotopy cofibre sequence
$
j_\sharp j^* (V/W) \to V/W \to i_* L^\Aeins i^* (V/W) 
$
from \eqref{glueing}.
We have $j_\sharp j^* (S/W)\simeq *$ and therefore $L^\Aeins(V/W)\simeq i_*L^\Aeins(i^*(V/W))$.
On the other hand, $i^*(S/W)\simeq i^*(S)/i^*(W)\simeq \sigma/\emptyset\simeq S^0_\sigma$ and hence $i_*L^\Aeins(i^*(V/W))$ has non-trivial $\pi_0$.
\end{remark}

Now we turn to assumption~\eqref{mainprop1} of Proposition~\ref{mainproposition}.
For the special case of $S$ the spectrum of a field, this is an observation of Morel in \cite[Lem.~3.3.6]{Morel04}.
Please note that the extra claim $s^*(W)\neq\emptyset$ in the next lemma will exclude the obvious obstacle to assumption~\eqref{mainprop2} of Proposition~\ref{mainproposition} in our application (\cf~Re\-mark~\ref{counterexampleV/W}).

\begin{lemma}\label{lemmatoconstructW}
Let $S$ be a noetherian scheme of finite Krull-dimension together with a codimension $c$ point $s\in S$.
Let $E\in\SH^\simpl_{S^1 > c}(S)$ be a spectrum.
Then, for any $V\in\Sm_S$ with $s^*(V)\neq\emptyset$ and any $f\in [\Sigma^\infty_{S^1} V_\pt, L^\Aeins E]$ there exists an open subscheme $W\hookrightarrow V$ with $f|_{\Sigma^\infty_{S^1} W_\pt}=0$ and $s^*(W)\neq\emptyset$.
\end{lemma}
\begin{proof}
Let $\eta_Z\in V$ be a generic point of an irreducible component $Z$ of $s^*(V)$.
In particular, the ring $\mathcal{O}_{V,\eta_Z}$ has dimension $c$.
We write
\[
\mathfrak{j}\colon\mathfrak{U}:=\Spec(\mathcal{O}_{V,\eta_Z})~\cong ~\underset{j_i\colon U_i\hookrightarrow V}{\lim} U_i \rightarrow V,
\]
where the limit on the right-hand side is indexed by the diagram constituted by the open immersions $j_i\colon U_i\hookrightarrow V$ with $U_i$ affine and $U_i\cap Z\neq\emptyset$.

Let $p\colon V \to S$ denote the structural morphism.
We have
\begin{align*}
& \underset{j_i\colon U_i\hookrightarrow V}{\colim} [\Sigma_{S^1}^\infty(U_i\to S)_\pt, L^\Aeins E]_{\Spt_{S^1}(S)}&&  \\ 
~\cong~ & \underset{j_i\colon U_i\hookrightarrow V}{\colim} [\Sigma_{S^1}^\infty p_\sharp ((U_i\to V)_\pt), L^\Aeins E]_{\Spt_{S^1}(S)}&&\\ 
~\cong~ & \underset{j_i\colon U_i\hookrightarrow V}{\colim} [p_\sharp \Sigma_{S^1}^\infty (U_i\to V)_\pt, L^\Aeins E]_{\Spt_{S^1}(S)}&&\\ 
~\cong~ & \underset{j_i\colon U_i\hookrightarrow V}{\colim} [\Sigma_{S^1}^\infty(U_i\to V)_\pt, p^*( L^\Aeins E)]_{\Spt_{S^1}(V)}&&\\ 
~\cong~ & \underset{j_i\colon U_i\hookrightarrow V}{\colim} [\Sigma_{S^1}^\infty(U_i\to V)_\pt, L^\Aeins(p^* E)]_{\Spt_{S^1}(V)}&& \text{(by Lemma~\ref{interchangebysmoothstable}.\eqref{stablelimitbasechange1})}\\ 
~\cong~ & \underset{j_i\colon U_i\hookrightarrow V}{\colim} [ j_{i,\sharp} \Sigma_{S^1}^\infty (U_i\to U_i)_\pt, L^\Aeins(p^* E)]_{\Spt_{S^1}(V)}&& \text{($j_{i,\sharp} ((U_i\to U_i)_\pt) = (U_i\to V)_\pt$)}\\ 
~\cong~ & \underset{j_i\colon U_i\hookrightarrow V}{\colim} [\Sigma_{S^1}^\infty (U_i\to U_i)_\pt, j_i^* L^\Aeins(p^* E)]_{\Spt_{S^1}(U_i)} &&\\
~\cong~ & [ \Sigma_{S^1}^\infty(\mathfrak{U}\to \mathfrak{U})_\pt , \mathfrak{j}^* L^\Aeins (p^* E) ]_{\Spt_{S^1}(\mathfrak{U})} &&\text{(by Lemma~\ref{interchangebysmoothstable}.\eqref{stablelimitbasechange2})}\\
~\cong~ & [ \Sigma_{S^1}^\infty(\mathfrak{U}\to \mathfrak{U})_\pt , L^\Aeins (\mathfrak{j}^*p^* E) ]_{\Spt_{S^1}(\mathfrak{U})} &&\text{(by Lemma~\ref{interchangebysmoothstable}.\eqref{stablelimitbasechange1})}.\\
\end{align*}
Using the Quillen adjoint pair $(p_\sharp, p^*)$, we see that $p^*$ preserves connectivity.
By Lemma \ref{interchangebysmoothstable}, the same is true for $\mathfrak{j}^*$, so $\mathfrak{j}^*(p^* E)$ is contained in $\SH^\simpl_{S^1> c}(\mathfrak{U})$.
By Proposition \ref{absoluteconnectivity}, we get $[\Sigma^\infty_{S^1}\mathfrak{U}_\pt,\mathfrak{j}^*p^* E] = 0$ as the scheme $\mathfrak{U}$ has dimension $c$.
\\
The restrictions of $f\in [\Sigma^\infty_{S^1} V_\pt, L^\Aeins E]$ induce an element of the set $\colim_{j_i}[\Sigma_{S^1}^\infty U_{i,\pt}, L^\Aeins E]=0$ from the left-hand side of the chain of equations above.
This means that there exists an open subscheme $W := U_i\hookrightarrow V$ with $W\cap Z\neq\emptyset$ and $f|_{\Sigma_{S^1}^\infty W_\pt}=0$. Since $Z\subseteq s^*(V)$, we have $s^*(W)\neq\emptyset$.
\end{proof}

Finally, let us mention that for connectivity results we may restrict ourself to local base schemes.

\begin{lemma}\label{lemmastalksofthebase}
Let $S$ be a noetherian scheme of finite Krull-dimension and let $r\geq 0$ be an integer.
Let $E\in\SH^s_{S^1\geq i}(S)$ be a spectrum.
Suppose that for all points $s\in S$ and $\mathfrak{s}\colon S_s^h\to S$, we have that $L^\Aeins \mathfrak{s}^*E\in \SH^s_{S^1\geq i-r}(S_s^h)$, then $L^\Aeins E\in \SH^s_{S^1\geq i-r}(S)$.
\end{lemma}
\begin{proof}
After shifting, we can assume that $i=r$.
We have to show that the sheaf $\pi_0^\Aeins(E)$ is trivial.
It follows from Corollary~\ref{localcorollary} that $\pi_0(L^\Aeins E)$ is trivial if and only if $\pi_0(\mathfrak{s}^* L^\Aeins E)$ is trivial for all $s\in S$.
Hence the claim follows.
\end{proof}

\subsection*{Shifted stable $\AA^1$-connectivity}

In order to obtain a uniform bound for the loss of connectivity of a spectrum $E$, we may restrict ourself to a local base scheme $S$ by Lemma~\ref{lemmastalksofthebase}.
We want to invoke Proposition~\ref{mainproposition}.
Given a $V$-section $f$ of a homotopy presheaf of $E$, we have to search for an open subscheme $W\hookrightarrow V$ fulfilling assumption~\eqref{mainprop1} and ~\eqref{mainprop2} of Proposition~\ref{mainproposition}, \ie, $f=0$ when restricted to $W$, and $V/W$ is $\AA^1$-Nisnevich-local connected.
For the former condition, we want to apply Lemma~\ref{lemmatoconstructW}.
To avoid the obstacle to the latter condition comming from Remark~\ref{counterexampleV/W}, we need to take care that $W$ has non-empty fibre over the closed point of $S$.
This point has codimension $c$ the dimension $d$ of $S$.
Thus, again by Lemma~\ref{lemmatoconstructW}, a natural candidate for a uniform bound on the loss of connectivity of $E$ is $c=d$.
This motivates the following definition respectively question.

\begin{defi}
\label{defi:shiftedstable}
A noetherian scheme $S$ of dimension $d$ has the \emph{shifted stable $\AA^1$-connectivity property}, if for every integer $i$ and every spectrum $E$ in $\SH^s_{S^1\geq i}(S)$, the $\AA^1$-localization $L^\Aeins E$ is contained in $\SH^s_{S^1\geq i-d}(S)$.
\end{defi}

\begin{question}\label{thequestion}
Let $S$ be a regular noetherian scheme of dimension $d$. Does $S$ have the shifted stable $\AA^1$-connectivity property?
\end{question}

\begin{remark}
Morel's connectivity theorem (\cf~Theorem \ref{morelconnectivity} above) provides a positive answer in the case of $S$ the spectrum of a field.
In the case of $S$ a Dedekind scheme with all residue fields infinite, we get a positive answer by Theorem~\ref{maintheorem} below.
Unfortunately, we do not have a positive or negative answer for more general base schemes. 
\end{remark}

\begin{remark}\label{counterexample2}
The example of Ayoub discussed in Remark~\ref{counterexample} does not provide a negative answer to Question~\ref{thequestion} above.
In fact, Ayoub gave an example of a base scheme $S$ (of dimension $\geq 2$) and a spectrum $E$ whose homotopy sheaves $\pi^\Aeins_i(E)$ are not \emph{strictly} $\AA^1$-invariant.
The latter property is a consequence of the \emph{non-shifted} stable $\AA^1$-connectivity property, \ie, the property that $\AA^1$-localization does not lower the connectivity at all.
However, the proof (\cf~\cite[Thm.~6.2.7]{Morel05}) that the non-shifted stable $\AA^1$-connectivity property implies strictly $\AA^1$-invariant homotopy sheaves does not carry over from the non-shifted to the shifted stable connectivity property of the base.
\end{remark}

At least, a positive answer to Question~\ref{thequestion} would follow from $\AA^1$-invariance of $\AA^1$-homotopy sheaves $\pi^\Aeins_k(E)$ as the following proposition shows.
Note that the property of a Nisnevich sheaf to be strictly $\AA^1$-invariant is a stronger property than just being $\AA^1$-invariant (\cf~Remark~\ref{counterexample2}).

\begin{proposition}\label{evidence}
Let $S$ be a noetherian scheme of dimension $d$.
Let $i$ be an integer and $E\in\SH^\simpl_{\geq i}(S)$ such that the sheaf $\pi^\Aeins_k(E)$ is $\AA^1$-invariant for all integers $k<i-d$.
Then $L^\Aeins E\in \SH^s_{S^1\geq i-d}(S)$.
\end{proposition}
\begin{proof}
First note, that for any open immersion $j\colon S'\hookrightarrow S$ the functor $j^*$ preserves $\AA^1$-invariance of (simplicial) presheaves and $j^*\pi_0\cong \pi_0j^*$.
In particular, our assumptions on the spectrum are stable under restriction to open subschemes of the base.
Let $E$ be a spectrum in $\SH^\simpl_{S^1>d}$ with $\AA^1$-invariant homotopy sheaves $\pi^\Aeins_k(E)$ in degrees $k\leq 0$.
To prove Proposition \ref{evidence}, it is again enough to show that $\pi^\Aeins_0(E)$ is trivial.
We argue by induction on the dimension $d$ of the base $S$.
The case $d=0$ is Theorem~\ref{morelconnectivity}. Let $d>0$.
By Corollary \ref{localcorollary}, we may assume that $S$ is local with closed point $i\colon\sigma\hookrightarrow S$.
Take a connected scheme $V\in\Sm_S$ with structure morphism $p\colon V\to S$ and a point $v\in V$.
It suffices to show that the Nisnevich-stalk of $\pi_0^\Aeins(E)$ at $(V,v)$ is trivial.
By the induction hypothesis, we may assume that $v$ lies in the fibre over $\sigma$ as the open complement $S\setminus\sigma$ has Krull-dimension strictly smaller than $d$.
Moreover, we may assume that $i^*(V)$ is connected.
Let $f_{(V,v)}$ be a germ in this stalk.
We have to show that $f_{(V,v)}$ is trivial:
After possibly refining $(V,v)$ Nisnevich-locally, we may assume that $f_{(V,v)}$ is induced by an element $f\in[\Sigma^\infty_{S^1}V_\pt, L^\Aeins E]$.
By Lemma \ref{lemmatoconstructW}, there exists an open subscheme $W\hookrightarrow V$ with $f|_{\Sigma^\infty_{S^1}W_\pt}=0$ and $i^*(W)\neq\emptyset$.
Clearly, we may assume that $v\notin W$. The cofibre sequence $W_\pt\to V_\pt \to V/W$ induces an exact sequence
\[
0\to \tilde\pi_0(L^\Aeins E)(V/W) \to \tilde\pi_0(L^\Aeins E)(V) \to \tilde\pi_0(L^\Aeins E)(W)
\]
of homotopy sheaves.
Here, we write $\tilde\pi_0(L^\Aeins E)(V/W)$ for $\Hom(V/W,\tilde\pi_0(L^\Aeins E))$.
Since the restriction of $f$ to $W$ is trivial, it suffices to show that $\tilde \pi_0(L^\Aeins E)(V/W)$ is trivial.
The $\AA^1$-Nisnevich-local homotopy cofibre sequence
\[
j_\sharp j^* (V/W) \to V/W \to i_* L^\Aeins i^* (V/W) 
\]
from ~\eqref{glueing} induces a long exact sequence
\[
\cdots\to [i_* L^\Aeins i^* (V/W), \pi_0^\Aeins(E)]\to [V/W, \pi_0^\Aeins(E)]\to [j_\sharp j^* (V/W),\pi_0^\Aeins(E)]
\]
by the $\AA^1$-Nisnevich-local fibrancy of $\pi_0^\Aeins(E)=\pi_0(L^\Aeins E)$.
For the latter, note that a sheaf considered as a discrete simplicial presheaf is Nisnevich-locally fibrant.
The set on the right-hand side equals $[j^* (V/W),j^*\pi_0^\Aeins(E)]$ and $j^* \pi_0^\Aeins(E)\cong \pi_0^\Aeins(j^*E)$ is trivial by induction.
The triviality of the set on the left-hand side follows from the triviality of $\pi_0(i_* L^\Aeins i^* (V/W))$.
By \cite[Prop.~4.2]{Spitzweck14}, the latter is zero, if $\pi^\Aeins_0(i^* (V/W))=0$.
Since $i^*(V)$ is irreducible and $i^*(W)$ is non-empty, we conclude by \cite[Lem.~6.1.4]{Morel05}.
\end{proof}

\subsection*{The one dimensional case}
Using the Gabber-presentation provided by Theorem~\ref{thm: gabber presentation}, we can give a positive answer to Question \ref{thequestion} for a Dedekind scheme $S$ with infinite residue fields.

\begin{theorem}\label{maintheorem}
Let $S$ be a Dedekind scheme and assume that all of its residue fields are infinite.
Then $S$ has the shifted stable $\AA^1$-connectivity property:
$E\in\SH^s_{S^1\geq i}(S)$ implies $L^\Aeins E\in \SH^s_{S^1\geq i-1}(S)$. 
\end{theorem}
\begin{proof}
By Lemma~\ref{lemmastalksofthebase}, we may assume that $S$ is henselian local of dimension $\leq 1$ with infinite residue field and closed point $\sigma$.
The case of dimension zero is covered by Theorem~\ref{morelconnectivity}.
Hence we may assume that $S$ is the spectrum of a henselian discrete valuation ring.
We want to apply Proposition~\ref{mainproposition}.
Consider an element $f\in[\Sigma^\infty_{S^1}V_\pt,L^\Aeins E]$ for $V\in\Sm_S$.
We may assume that $\sigma^*(V)\neq\emptyset$ since otherwise we argue as Morel in the proof of Theorem~\ref{morelconnectivity}.
By Lemma~\ref{lemmatoconstructW} applied to the closed point $\sigma$ of $S$, we find an open subscheme $W\hookrightarrow V$ such that $f_{\Sigma^\infty_{S^1}W_\pt}=0$ and $\sigma^*(W)\neq\emptyset$.
Let $i\colon Z\hookrightarrow V$ be the reduced closed complement of $W$.
In particular, $Z_\sigma\neq V_\sigma$.
By Theorem~\ref{thm: gabber distinguished square}, the conditions of Lemma~\ref{connectivityassuminggabber} are fulfilled and we get the second assumption $\pi^\Aeins_0(V/W)=0$ of Proposition~\ref{mainproposition}, as well.
\end{proof}

\providecommand{\bysame}{\leavevmode\hbox to3em{\hrulefill}\thinspace}
\providecommand{\href}[2]{#2}

\end{document}